\documentclass[sn-mathphys]{sn-jnl}% Math and Physical Sciences Reference Style
\jyear{2021}%

\theoremstyle{thmstyleone}%
\newtheorem{theorem}{Theorem}%  meant for continuous numbers
\theoremstyle{thmstyletwo}%

\theoremstyle{thmstylethree}%

\usepackage{amssymb}
\usepackage{epstopdf}
\usepackage{array,amsfonts,amsmath,amsthm,amssymb,graphicx,relsize,url,array,indentfirst,booktabs,geometry}

\usepackage{anyfontsize}

\def\<{{\langle}}
\def\>{{\rangle}}

\def\R{{\mathbb{R}}}

\def\N{{\mathbb{N}}}

\def\W{{\mathbf{W}}}

\def\F{{\mathcal{F}}}

\newcommand{\bs}[1]{\boldsymbol{#1}}
\newcommand{\bit}{\begin{itemize}}
\newcommand{\eit}{\end{itemize}}

\newcommand{\ipar}{m}

\newcommand{\ParSpace}{\mathcal{M}}
\newcommand{\bParSpace}{\mathcal{M}_n}

\newcommand{\mupr}{\mu_{\text{pr}}}
\newcommand{\mupost}{\mu_{\text{post}}}

\newcommand{\Cpr}{\mathcal{C}_{\text{pr}}}

\newcommand{\pilike}{\pi_{\text{like}}}

\newcommand{\Gnoise}{\mathbf{\Gamma}_{\text{n}}}

\newcommand{\obs}{\mathbf{y}}

\newcommand{\ObsSpace}{\mathcal{Y}}
\newcommand{\obsn}{\boldsymbol{\varepsilon}}

\newcommand{\ave}{\mathbb{E}}

\newcommand{\bipar}{\mathbf{\ipar}}
\newcommand{\bF}{\mathbf{F}}

\newcommand{\bmupr}{\mathbf{m}_{\text{pr}}}
\newcommand{\bCpr}{\mathbf{\Gamma}_{\text{pr}}}

\def\mpr{{m_{\text{pr}}}}

\usepackage{xcolor}

\newcommand{\kw}[1]{{\color{black}{#1}}}
\newcommand{\tom}[1]{{\color{black}{#1}}}
\renewcommand{\normalsize}{\fontsize{10pt}{\baselineskip}\selectfont}
\usepackage{textcomp,fancyhdr,lastpage}
\usepackage{extarrows}
\usepackage[utf8]{inputenc}
\DeclareMathOperator*{\argmax}{arg\,max}

\newcommand\blfootnote[1]{%
  \begingroup
  \renewcommand\thefootnote{}\footnote{#1}%
  \addtocounter{footnote}{-1}%
  \endgroup
}

\begin{document}

\title[
  Large-scale Bayesian optimal experimental design with DIPNet]{\kw{Large-scale Bayesian optimal experimental design with derivative-informed projected neural network}}

\author*[1]{\fnm{Keyi} \sur{Wu}}\email{keyiwu.w@gmail.com}

\author[2]{\fnm{Thomas} \sur{O'Leary-Roseberry}}\email{tom.olearyroseberry@utexas.edu}

\author[3]{\fnm{Peng} \sur{Chen}}\email{peng@cc.gatech.edu}

\author[2,4]{\fnm{Omar} \sur{Ghattas} \blfootnote{This research was partially funded by DOE ASCR 
  DE-SC0019303 and DE-SC0021239, DOD MURI FA9550-21-1-0084, 
  and NSF DMS-2012453.}}\email{omar@oden.utexas.edu}  %\equalcont{These authors contributed equally to this work.}

\affil*[1]{\orgdiv{Department of Mathematics}, \orgname{The University
    of Texas at Austin}, \orgaddress{\street{2515 Speedway},
    \city{Austin}, \state{TX} \postcode{78712}, \country{USA}}}

\affil[2]{\orgdiv{Oden Institute for Computational Engineering and
    Sciences}, \orgname{The University of Texas at Austin},
  \orgaddress{\street{201 E 24th St}, \city{Austin}, \postcode{78712},
    \state{TX}, \country{USA}}}

\affil[3]{\orgdiv{School of Computational Science and Engineering}, \orgname{Georgia Institute of Technology},
  \orgaddress{ \city{Atlanta}, \postcode{30308},
    \state{GA}, \country{USA}}}

\affil[4]{\orgdiv{Departments of Geological Sciences and Mechanical
    Engineering}, \orgname{The University of Texas at Austin},
  \orgaddress{
%  \street{201 E 24th St},
\city{Austin}, \postcode{78712}, \state{TX}, \country{USA}}}


\abstract{We address the solution of large-scale Bayesian optimal
  experimental design (OED) problems governed by partial differential
  equations (PDEs) with infinite-dimensional parameter fields.  The
  OED problem seeks to find sensor locations that maximize the
  expected information gain (EIG) in the solution of the underlying
  Bayesian inverse problem. Computation of the EIG is usually
  prohibitive for PDE-based OED problems. To make the evaluation of
  the EIG tractable, we approximate the (PDE-based)
  parameter-to-observable map with a derivative-informed projected
  neural network (DIPNet) surrogate, which exploits the geometry,
  smoothness, and intrinsic low-dimensionality of the map using a
  small and dimension-independent number of PDE solves. The surrogate
  is then deployed within a greedy algorithm-based solution of the OED
  problem such that no further PDE solves are required. We analyze the
  EIG approximation error in terms of the generalization error of the
  DIPNet \tom{and show they are of the same order.} Finally, the efficiency
  and accuracy of the method are demonstrated via numerical
  experiments \tom{on OED problems governed by inverse scattering and
  inverse reactive transport with up to 16,641 uncertain parameters
  and 100 experimental design variables, where we observe up to three
  orders of magnitude speedup relative to a reference double loop Monte
  Carlo method. }}




\maketitle

%%================================%%
%% Introduction %%
%%================================%%
\section{Introduction}\label{sec:intro}

In modeling natural or engineered systems, uncertainties are often
present due to the lack of knowledge or intrinsic variability of the
system. Uncertainties may arise from sources as varied as initial and
boundary conditions, material properties and other coefficients,
external source terms, interaction and coupling terms, and geometries;
for simplicity of exposition, we refer to all of these as {\em
  parameters}.  Uncertainties in prior knowledge of these parameters
can be reduced by incorporating indirect observational or experimental
data on the system state or related quantities of interest into the
forward model via solution of a Bayesian inverse problem. 
Prior knowledge can be incorporated through a prior distribution on
the uncertain parameters. The data are typically noisy because of
limited measurement precision, which induces a likelihood of the data
conditioned on the given parameters. Uncertainties of the parameters
are then reduced by the data and quantified by a posterior
distribution, which is a joint distribution of the prior and the
likelihood, given by Bayes' rule.

% Reducing the uncertainty from the prior to the posterior informed by
% the data is known as Bayesian inverse problems.

Large amounts of informative data can reduce  uncertainties in the
parameters, and thus posterior predictions, significantly. However,
the data are often sparse or limited due
%In many practical problems, we use experimental data to gain
%information, identify the uncertainty of a quantity of interest.  
to the cost of their acquisition. 
In such cases it is critical to design the acquisition process or
experiment in an optimal way so that as much information as possible
can be gained from the acquired data, or the uncertainty in the
parameters or posterior predictions can be reduced as much as
possible. Experimental design variables can include what, when, and
where to measure,  which sources to use to excite the system, and
under which conditions should the experiments be conducted. 
This is known as the optimal experimental design (OED) problem
\cite{Ucinski05}, or
Bayesian OED in the context of Bayesian inference. OED problems arise
across numerous fields including geophysical exploration, medical
imaging, nondestructive evaluation, drug testing, materials
characterization, and earth system data assimilation, to name just a few.
For example, two notable uses of OED include
optimal observing system design in oceanography
\cite{LooseHeimbach21} and optimal sensor placement for tsunami early warning
 \cite{FerrolinoLopeMendoza20}. 

%and optimal process design to accelerate MRI imaging  \cite{BakkerHoofWelling20}. 
The challenges to solving OED problems in these and other fields are
manifold. The models underlying the systems of interest typically take
the form of partial differential  
equations (PDEs) and can be large-scale, complex, nonlinear, dynamic,
multiscale, and coupled. 
The uncertain parameters may depend on both space and time, and are
often characterized by infinite-dimensional random fields and/or
stochastic processes. The PDE models can be extremely expensive to
solve for each realization of the infinite-dimensional uncertain
parameters. The computation of the OED objective involves
high-dimensional (after discretization) integration with respect to
(w.r.t.) the uncertain parameters, and thus require a large number of
PDE solves. Finally, the OED objective will need to be evaluated
numerous times, especially when the experimental design variables are
high-dimensional or when they represent discrete decisions.

To address these computational challenges, different classes of
methods have been developed by exploiting (1) sparsity via polynomial
chaos approximation of parameter-to-observable  maps
\cite{HuanMarzouk13,HuanMarzouk14,HuanMarzouk16},
(2) Laplace approximation of the posterior
\cite{AlexanderianPetraStadlerEtAl16,
BeckDiaEspathEtAl18, LongScavinoTemponeEtAl13a, LongMotamedTempone15,
BeckMansourDiaEspathEtAl20}, (3) intrinsic low
dimensionality by low-rank approximation of (prior-preconditioned and
data-informed) operators \cite{AlexanderianPetraStadlerEtAl14,
  AlexanderianGloorGhattas16,AlexanderianPetraStadlerEtAl16,
  SaibabaAlexanderianIpsen17,CrestelAlexanderianStadlerEtAl17,AttiaAlexanderianSaibaba18}, 
(4) decomposibility by offline (for PDE-constrained
approximation)--online (for design optimization) decomposition
\cite{WuChenGhattas20,WuChenGhattas21}, and (5) surrogate models of
the PDEs, parameter-to-observable map, or posterior distribution by model
reduction \cite{Aretz-NellesenChenGreplEtAl20, AretzChenVeroy21} 
and deep learning \cite{FosterJankowiakBinghamEtAl19,
  KleinegesseGutmann20,ShenHuan21}. 

Here, we consider the Bayesian OED problem for optimal 
sensor placement governed by large-scale and possibly
nonlinear PDEs with infinite-dimensional uncertain parameters. We use
the expected information gain (EIG), also known as mutual information,
as the optimality criterion for the OED problem. The optimization
problem is combinatorial: we seek the combination of sensors, selected
from a set of candidate locations, that maximizes the EIG. 
The EIG is an average of the Kullback--Leibler (KL) divergence between
the posterior and the prior distributions over all realizations of the
data. This involves a double integral: one integral of the likelihood
function w.r.t.\ the prior distribution to compute the normalization
constant or model evidence for each data realization, and one integral
w.r.t.\ the data distribution. To evaluate the two integrals we adopt
a double-loop Monte Carlo (DLMC) method that requires the computation
of the parameter-to-observable map at each of the parameter and data
samples. Since the likelihood can be rather complex and highly  locally
supported in the parameter space, the number of parameter samples from
the prior distribution needed to capture the
likelihood well with relatively accurate sample average approximation of
the normalization constant can be extremely large. 
The requirement to evaluate the
PDE-constrained parameter-to-observable map at each of the large
number of samples leads to numerous PDE solves, which is prohibitive
when the PDEs are expensive to solve. 
To tackle this challenge, we construct a derivative-informed projected
neural network (DIPNet)
\cite{OLearyRoseberryVillaChenEtAl2022,OLearyRoseberry2020,OLearyRoseberryDuChaudhuriEtAl2021}
surrogate of the parameter-to-observable map that exploits the
intrinsic low dimensionality of both the parameter and the data
spaces. This intrinsic low dimensionality is due to the correlation of
the high-dimensional parameters, the smoothing property of the
underlying PDE solution, and redundant information contained in the
data from all of the candidate sensors. In particular, the
low-dimensional subspace of the parameter space can be detected via
low rank approximations of derivatives of the parameter-to-observable
map, such as the 
Jacobian, Gauss-Newton Hessian, or higher-order derivatives. \tom{This
property has been observed and exploited across a wide spectrum of
Bayesian inverse problems
% and used in many uncertainty quantification
% problems such as model reduction for sampling and deep learning
% \cite{BashirWillcoxGhattasEtAl08,ChenGhattas19a,AlgerChenGhattas20},
% optimization under uncertainty
% \cite{AlexanderianPetraStadlerEtAl17,ChenVillaGhattas19,
%  ChenHabermanGhattas21},
\cite{FlathWilcoxAkcelikEtAl11,Bui-ThanhBursteddeGhattasEtAl12,
  Bui-ThanhGhattasMartinEtAl13,Bui-ThanhGhattas14,KalmikovHeimbach14,HesseStadler14,IsaacPetraStadlerEtAl15,CuiLawMarzouk16,ChenVillaGhattas17,BeskosGirolamiLanEtAl17,ZahmCuiLawEtAl18,BrennanBigoniZahmEtAl20,ChenWuChenEtAl19,ChenGhattas20,SubramanianScheufeleMehlEtAl20,BabaniyiNicholsonVillaEtAl21}
and Bayesian optimal experimental design
\cite{AlexanderianPetraStadlerEtAl16,WuChenGhattas20,WuChenGhattas21}. 
See
\cite{GhattasWillcox21} for analysis of model elliptic, parabolic, and
hyperbolic problems, and a lengthy list of complex inverse problems that have
been found numerically to exhibit this property. }

This intrinsic low-dimensionality of parameter and data spaces, along
with smoothness of the parameter-to-observable map, allow us to
construct an accurate (over parameter space) DIPNet surrogate with a
limited and dimension-independent number of training data pairs, each
requiring a PDE solve. Once trained, the DIPNet surrogate 
is deployed in the OED problem, which is solved without further PDE
solution, resulting in very large reductions in computing time. 
Under suitable assumptions, we provide an analysis of the error
propagated from the DIPNet approximation to the approximation of the
normalization constant and the EIG. To solve the combinatorial
optimization problem of sensor selection, we use a greedy algorithm
developed in our previous work
\cite{WuChenGhattas20,WuChenGhattas21}. We demonstrate the efficiency
and accuracy of our computational method by conducting two numerical
experiments with infinite-dimensional parameter fields: OED for inverse
scattering (with an acoustic Helmholtz
forward problem) and inverse reactive transport (with a nonlinear
advection-diffusion-reaction forward problem).

The rest of the paper is organized as follows. The setup of the
problems including Bayesian inversion, EIG, sensor design matrix, and
Bayesian OED are presented in Section \ref{sec:bg}. Section
\ref{sec:method} is devoted to presentation of the computational
methods including DLMC, DIPNet and its induced error analysis, and
the greedy optimization algorithm. Results for the
two OED numerical experiments are provided in Section \ref{sec:numerical}, followed
by conclusions in Section \ref{sec:conclusions}.

\section{Problem setup}\label{sec:bg}

%%================================%%
%% Bayesian inverse problem %%
%%================================%%

\subsection{Bayesian inverse problems}\label{sec:bayesian}
Let $\mathcal{D} \subset \mathbb{R}^{n_x}$ denote a physical domain in
dimension $n_x = 1,2,3$.  We consider the problem of inferring an
uncertain parameter field $\ipar$ defined in the physical domain
$\mathcal{D}$ from noisy data $\obs$ and a complex model represented by
PDEs. Let $\obs \in \R^{n_y}$ denote the noisy data vector of dimension $n_y \in \N$, given by
\begin{equation}
    \obs = \mathcal{F}(m) + \obsn,
\end{equation}
which is contaminated by the additive Gaussian noise $\obsn \sim \mathcal{N}(\bs{0},\Gnoise)$ with zero mean and covariance $\Gnoise \subset \R^{n_y\times n_y}$. Specifically,  
$\obs$ is obtained from observation of the solution of the PDE model at $n_y$ sensor locations. $\mathcal{F}$ is the parameter-to-observable map which depends on the solution of the PDE and an observation operator that extracts the solution values at the $n_y$ locations. 
% defined in a physical domain $\mathcal{D} \subset \mathbb{R}^{n_x}$, where $n_x = 1,2,3$. 

We consider the above inverse problem in a Bayesian framework. First, we assume that $m$ lies in an infinite-dimensional real separable Hilbert space $\ParSpace$, e.g., $\ParSpace = \mathcal{L}^2(\mathcal{D})$ of square integrable functions defined in $\mathcal{D}$.
Moreover, we assume that $m$ follows a Gaussian prior measure $ \mupr = \mathcal{N}(\mpr,\Cpr)$ with mean $\mpr \in \ParSpace$ and covariance operator $\Cpr$, a strictly positive, self-adjoint, and trace-class operator. As one example, we consider $\Cpr = \mathcal{A}^{-2}$, where $\mathcal{A} = -\gamma \Delta + \delta I$ is a Laplacian-like operator with prescribed homogeneous Neumann boundary condition, with Laplacian $\Delta$, identity $I$, and positive constants $\gamma, \delta > 0$; see \cite{Stuart10,Bui-ThanhGhattasMartinEtAl13,PetraMartinStadlerEtAl14} for more details. Given the Gaussian observation noise, the likelihood of the data $\obs$ for the parameter $m \in \ParSpace$ satisfies 
\begin{equation}
\label{eq:likelihood}
\pilike(\obs \vert \ipar) \propto \exp\left( -  \Phi(\ipar, \obs) \right),
\end{equation}
where 
\begin{equation}
\Phi(\ipar, \obs) := \frac{1}{2} \|\obs - \F(\ipar)\|^2_{\Gnoise^{-1}}
\end{equation}
is known as a potential function. By Bayes' rule, 
the posterior measure, denoted as $\mupost (\ipar \vert  \obs)$, is given by the Radon-Nikodym derivative as
\begin{equation}\label{eq:Bayes}
\frac{d \mupost (\ipar \vert \obs) }{d \mupr(\ipar)} = \frac{1}{\pi(\obs)}\pilike(\obs\vert \ipar),
\end{equation}
% where $\frac{d \mupost (\ipar \vert \obs) }{d \mupr(\ipar)}$ is the Radon-Nikodym derivative, 
where $\pi(\obs)$ is the so-called normalization constant or model evidence, given by 
\begin{equation}
\pi(\obs) = \int_\mathcal{M} \pilike(\obs\vert\ipar) d \mupr(\ipar).
\end{equation}
This expression is often computationally intractable because of the infinite-dimensional integral, which involves a (possibly large-scale) PDE solve for each realization $m$.

%%================================%%
%% Expected information gain optimal criterion %%
%%================================%%

\subsection{Expected information gain} \label{sec:eig}
To measure the information gained from the data $\obs$ in the inference of the parameter $m$, 
we consider a Kullback--Leibler (KL) divergence between the posterior and the prior, defined as 
\begin{equation}\label{eq:KL}
D_{\text{KL}}(\mupost (\cdot \vert \obs) \| \mupr ) := \int_{\ParSpace} \ln\left(
\frac{d\mupost(\ipar \vert \obs)}{d\mupr(\ipar)}\right) d\mupost(\ipar \vert \obs),
\end{equation}
which is random since the data $\obs$ is random. We consider a widely used optimality criterion,
expected information gain (EIG), which is the KL divergence averaged over all realizations of the data, defined as 
\begin{equation}\label{eq:EIG}
\begin{split}
\Psi 
&:= \ave_\obs \left[D_{\text{KL}}(\mupost(\cdot \vert
                 \obs)\| \mupr)\right] \\  
& = \int_{\ObsSpace} D_{\text{KL}}(\mupost(\cdot \vert \obs) \| \mupr )
       \, \pi(\obs)\, d\obs \\ 
& =\int_{\ObsSpace} \int_{\ParSpace} \ln\left(
\frac{d\mupost(\ipar \vert \obs)}{d\mupr(\ipar)}\right) d\mupost(\ipar \vert \obs)\,\pi(\obs)\,  d\obs\\
& = \int_{\ParSpace} \int_{\ObsSpace} \ln\left(
\frac{\pilike(\obs \vert \ipar )}{\pi(\obs)}\right) \pilike(\obs \vert \ipar )\,d \obs \, d\mupr(\ipar),
\end{split}
\end{equation}
where the last equality follows Bayes' rule \eqref{eq:Bayes} and the Fubini theorem under the assumption of proper integrability.

%%================================%%
%% Optimal perimental design problems in Bayesian inverse problems %%
%%================================%%

\subsection{Optimal experimental design} \label{sec:oed}

We consider the OED problem for optimally acquiring data %determining how to collect measurement data 
to maximize the expected information gained in the parameter inference. The experimental of design seeks 
% to choose a subset out of all possible candidates. 
% Examples of the candidates include observation locations to infer the permeability field for subsurface flow and incident angles to infer material property by X-ray scattering, 
% and power spectra for image data.
% \blue{add some references?}
% no need to mention these examples to avoid ambiguity
% For simplicity and to be consistent with the numerical examples we present, we will talk about sensor locations as experiment designs in our work. Assume we want 
to choose $r$ sensor locations out of $d$ candidates $\{\bs{x}_1,\dots,\bs{x}_d \}$ 
% to optimize the data collection, 
represented by a design matrix $\W \in \R^{r\times d} \in \mathcal{W}$,
% to represent a sensor selection. 
namely, if the $i$-th sensor is placed at $\bs{x}_j$, then $\W_{ij} = 1$, otherwise $\W_{ij} = 0$:
\begin{equation}
   \mathcal{W} := \left\{ \W \in \R^{r \times d}: \W_{ij} \in \{0,1 \}, \sum^r_{i=1} \W_{ij} \in \{0,1 \}, \sum^d_{j=1} \W_{ij} = 1 \right\}.
\end{equation}
Let $\F_d: \ParSpace \mapsto \R^{d}$ denote the parameter-to-observable map and $\obsn_d \in \R^d$ denote the additive noise, both using all $d$ candidate sensors; then we have 
\begin{equation}
\label{eq:WF}
    \F = \W \F_d \quad \text{and} \quad \obsn = \W \obsn_d.
\end{equation}
Then the likelihood (\ref{eq:likelihood}) for a specific design $\W$ is given by 
\begin{equation}
\label{eq:likeliW}
\pilike(\obs \vert \ipar,\W) \propto \exp\left( -  \frac{1}{2} \| \obs - \W\F_d(\ipar)\|^2_{\Gnoise^{-1}} \right),
\end{equation}
and the normalization constant also depends on $\W$ as
\begin{equation}
   \pi(\obs, \W) = \int_{\ObsSpace}  \pilike(\obs \vert m, \W) d \mupr(\ipar).
\end{equation}

From Section \ref{sec:eig}, we can see that the EIG $\Psi$ depends on the design matrix $\W$ through the likelihood function $\pilike(\obs \vert m,\W)$. To this end, we formulate the OED problem to find an optimal design matrix $\W^*$ such that
\begin{equation}\label{eq:OED}
   \W^* =  \argmax_{\W \in \mathcal{W}} \Psi(\W),
\end{equation}
with the $\W$-dependent EIG given by 
\begin{equation}
    \Psi(\W)
   = \int_{\ParSpace} \int_{\ObsSpace} \ln\left(
\frac{\pilike(\obs \vert \ipar,\W )}{\pi(\obs,\W)}\right) \pilike(\obs \vert \ipar,\W )\,d \obs \, d\mupr(\ipar).
\end{equation}
% with $\pilike(\obs \vert m,\W)$ defined in (\ref{eq:likeliW}).

%%================================%%
%% Discretization of the OED problem %%
%%================================%%

\subsection{Finite-dimensional approximation 
% Parameter discretization 
% of the OED problem
} \label{sec:Discretization}
To facilitate the presentation of our computational methods, we make a finite-dimensional approximation of the parameter field 
% discretize the parameter 
% We use a 
by using a finite element discretization. 
% to approximate infinite-dimensional Hilbert space $\ParSpace$ in a subspace $\bParSpace$.
Let 
$\bParSpace \subset \ParSpace$ denote a subspace of $\ParSpace$ spanned by  $n$ piecewise continuous Lagrange polynomial basis functions $\{ \psi_j\}^{n}_{j=1}$ over a mesh with elements of size $h$. Then the discrete parameter $\ipar_h \in \bParSpace$ is given by 
\begin{equation}
  \ipar_h =  \sum_{j=1}^{n} \ipar_j \psi_j .
\end{equation}
The Bayesian inverse problem is then stated for the finite-dimensional coefficient vector $\bipar = (\ipar_1,\dots, \ipar_n)^T$ of $\ipar_h$, with $n$ possibly very large. 
% The discrete parameter $\bipar \in \R^n$ is equipped with finite element mass matrix inner product. The discretized parameter-to-observable map $\bF: \R^n \to \R^{n_y}$ has its adjoint  $\bF^*$ given by $\bF^* = \mathbf{M}^{-1}\bF^T$, where $\mathbf{M}$ is the mass matrix. The corresponding discretized Bayesian inverse problem gives as follows.
The prior distribution of the discrete parameter $\bipar$ is Gaussian  $\mathcal{N} (\bmupr,\bCpr)$ with $\bmupr$ representing the coefficient vector of the discretized prior mean of $\mpr$, and $\bCpr$ representing the covariance matrix corresponding to $\Cpr = \mathcal{A}^{-2}$, given by 
\begin{equation}
  \bCpr =  \mathbf{A}^{-1}\mathbf{M}\mathbf{A}^{-1},
\end{equation}
where $\mathbf{A}$ is the finite element matrix of the Laplacian-like operator $\mathcal{A}$, and $\mathbf{M}$ is the mass matrix. Moreover, let $\bF_d: \R^n \to \R^{d}$ denote the discretized parameter-to-observable map corresponding to $\mathcal{F}_d$, we have $\bF = \W \bF_d$ as in (\ref{eq:WF}). Then 
the likelihood function corresponding to (\ref{eq:likeliW}) for the discrete parameter $\bipar$ is given by
\begin{equation}
\label{eq:blikelihood}
\pilike(\obs \vert \bipar, \W) \propto \exp\left( -  \frac{1}{2} \|\obs - \W\bF_d(\bipar)\|^2_{\Gnoise^{-1}} \right).
\end{equation}

%%================================%%
%% Method description %%
%%================================%%

\section{Computational methods} \label{sec:method}

%%================================%%
%% Double-loop Monte Carlo estimator%%
%%================================%%

\subsection{Double-loop Monte Carlo estimator} \label{sec:dlmc}
To solve the OED problem \eqref{eq:OED}, we need to evaluate the EIG repeatedly for each given design $\W$. The double integrals in the EIG expression can be computed 
by a double-loop Monte Carlo (DLMC) estimator $\Psi^{dl}$ defined as
\begin{equation}\label{eq:Psi-dl}
    \Psi^{dl}(\W) := \frac{1}{n_{\text{out}}} \sum^{n_{\text{out}}}_{i=1}\log\left(\frac{\pilike(\obs_i\vert\bipar_i, \W)}{\hat{\pi}(\obs_i, \W)}\right),
\end{equation}
where $\bipar_i$, $i = 1, \dots, n_{\text{out}}$, are i.i.d.\ samples from prior $\mathcal{N}(\bipar_{\text{pr}}, \bCpr)$ in the outer loop and $\obs_i = \mathbf{F}(\bipar_i) + \obsn_i$ are the realizations of the data with i.i.d.\ noise $\obsn_i \sim \mathcal{N}(\bs{0},\Gnoise)$.
$\hat{\pi}(\obs_i, \W)$ is a Monte Carlo estimator of the normalization constant $\pi(\obs_i,  \W)$ with $n_{\text{in}}$ samples in the inner loop, given by 
\begin{equation}\label{eq:hat-pi}
    \hat{\pi}(\obs_i,  \W) := \frac{1}{n_{\text{in}}}\sum^{n_{\text{in}}}_{j=1}\pilike(\obs_i\vert{\bipar}_{i,j},  \W),
\end{equation}
where ${\bipar}_{i,j}$, $j = 1, \dots, n_{\text{in}}$, are i.i.d.\ samples from the prior $\mathcal{N}(\bipar_{\text{pr}}, \bCpr)$.

For complex posterior distributions, e.g., high-dimensional, locally supported, multi-modal, non-Gaussian, etc., evaluation of the normalization constant is often intractable, i.e., a prohibitively large number of samples $n_{\text{in}}$ is needed.
As one particular example, when the posterior of $\bipar$ for data $\obs_i$ generated at sample $\bipar_i$ concentrates in a very small region far away from the mean of the prior, the likelihood $\pilike(\obs_i \vert \bipar_{i,j}, \W)$ is extremely small for most samples $\bipar_{i,j}$, 
% which can cause the value of $\hat{\pi}(\obs_i)$ to be numerically zero. 
which which leads to a requirement of a large number of samples to evaluate $\hat{\pi}(\obs_i, \W)$ with relatively small estimation error. This is usually prohibitive, since one evaluation of the parameter-to-observable map, and thus one solution of the large-scale PDE model, is required for each of $n_{\text{out}} \times n_{\text{in}}$ samples. This $n_{\text{out}} \times n_{\text{in}}$ PDE solves are required for each design matrix $\W$ at each optimization iteration.

% The dominant computational cost of DLMC estimator lies in the Monte Carlo estimation of normalization constant $\hat{\pi}(\obs_i)$ for each experimental data $\obs_i$ as each sample ${\bipar}_{i,j}$ requires a parameter-to-observable evaluation. In total, it need $n_{\text{out}} \times n_{\text{in}}$ parameter-to-observable evaluations. 

%%================================%%
%% {Derivative-informed projected neural networks for high-dimensional parametric maps%%
%%================================%%

\subsection{Derivative-informed projected neural networks 
% for high-dimensional parametric maps 
} \label{sec:DIPNet}

Recent research has motivated the deployment of neural networks as surrogates for parametric PDE mappings \cite{BhattacharyaHosseiniKovachki2020,FrescaManzoni2022,KovachkiLiLiuEtAl2021,LiKovachkiAzizzadenesheliEtAl2020b,LuJinKarniadakis2019,OLearyRoseberryChenVillaEtAl2022,NelsenStuart21,NguyenBui2021model,OLearyRoseberryVillaChenEtAl2022}. These surrogates can be used to accelerate the computation of the EIG within OED problems.
Specifically, to reduce the prohibitive computational cost, we build a surrogate for the parameter-to-observable map $\mathbf{F}_d: \R^n \mapsto \R^d$ at all candidate sensor locations by a derivative-informed projected neural network (DIPNet) \cite{OLearyRoseberry2020,OLearyRoseberryDuChaudhuriEtAl2021,OLearyRoseberryVillaChenEtAl2022}. Often, PDE-constrained high-dimensional parametric maps, such as the parameter-to-observable map $\mathbf{F}_d$, admit low-dimensional structure due to the correlation of the high-dimensional parameters, the regularizing property of the underlying PDE solution, and/or redundant information in the data from all candidate sensors.
% due to discretization effects. 
When this is the case, the DIPNet can exploit this low-dimensional
structure and parametrize a parsimonious map between the most informed
subspaces of the input parameter and the output observables. The
dimensions of the input and output subspaces are referred to as the
``information dimension'' of the map, which is often significantly
smaller than the parameter and data dimensions. \tom{The architectural
  strategy that we employ exploits compressibility of the map, by
  first reducing the input and output dimensionality via projection to
  informed reduced bases of the inputs and outputs. A neural network
  is then used to construct a low-dimensional nonlinear mapping
  between the two reduced bases. Error bounds for the effects of basis
  truncation, and parametrization by neural network are studied in
  \cite{BhattacharyaHosseiniKovachki2020,OLearyRoseberryDuChaudhuriEtAl2021,OLearyRoseberryVillaChenEtAl2022}.}

For the input parameter dimension reduction, we use a vector
generalization of an active subspace (AS)
\cite{ZahmConstantinePrieurEtAl2020}, which is spanned by the
generalized eigenvectors (input reduced basis) corresponding to the
$r_M$ largest eigenvalues of the eigenvalue problem 
\begin{equation} \label{as_evp}
	\left[\int_{\mathcal{M}_n}\nabla_{\bipar}\bF_d(\bipar) ^T\nabla_{\bipar}\bF_d(\bipar ) d\mu_{\text{pr}}(\bipar)\right]\mathbf{v}_i = \lambda_i^{AS} \bCpr^{-1} \mathbf{v}_i,
\end{equation}
where the eigenvectors $\mathbf{v}_i$ are ordered by the decreasing generalized eigenvalues $\lambda_i^\text{AS}$, $i = 1, \dots, r_M$. 
%An informed reduced basis for the output can be computed via 
For the output data dimension reduction, we use a proper orthogonal decomposition (POD)\cite{ManzoniNegriQuarteroni2016,QuarteroniManzoniNegri2015}, which uses the eigenvectors (output reduced basis) corresponding to the first $r_F$ eigenvalues of the expected observable outer product matrix,
\begin{equation} \label{pod_evp}
	\left[\int_{\mathcal{M}_n}\bF_d(\bipar)\bF_d(\bipar)^T d\mu_{\text{pr}}(\bipar)\right] \mathbf{\phi}_i = \lambda_i^{POD}  \mathbf{\phi}_i,
\end{equation}
where the eigenvectors $\mathbf{\phi}_i$ are ordered by the decreasing eigenvalues $\lambda_i^\text{POD}$, $i = 1, \dots, r_F$. When the eigenvalues of AS and POD both decay quickly, the mapping $\bipar \mapsto \bF_d(\bipar)$ can be well approximated when $\bipar$ and $\bF_d$ are projected to the corresponding subspaces with small $r_M$ and $r_F$; \tom{in this case approximation error bounds for reduced basis representation of the mapping are given by the trailing eigenvalues of the systems \eqref{as_evp},\eqref{pod_evp}}. \tom{This allows one to detect appropriate ``breadth'' for the neural network via the direct computation of the associated eigenvalue problems, removing the need for ad-hoc neural network hyperparameter search for appropriate breadth}. The neural network surrogate $\tilde{\bF}_d$ of the map $\mathbf{F}_d$ then has the form
\begin{equation}
	\tilde{\bF}_d(\bipar,[\mathbf{w},\mathbf{b}]) = \mathbf{\Phi}_{r_F}\mathbf{f}_r(\mathbf{V}_{r_M}^T\bipar,\mathbf{w}) + \mathbf{b},
\end{equation}
where $\mathbf{\Phi}_{r_F} \in \R^{d\times r_F}$ represents the POD reduced basis for the output, $\mathbf{V}_{r_M} \in \R^{n\times r_M}$ represents the AS reduced basis for the input, $\mathbf{f}_r\in\R^{r_F}$ is the neural network mapping between the two bases parametrized by weights $\mathbf{w}$ and bias $\mathbf{b}$. \tom{Since the reduced basis dimensions $r_F,r_M$ are chosen based on spectral decay of the AS and POD operators, we can choose them to be the same; for convenience we denote the reduced basis dimension instead by $r$.} The remaining difficulty is how to properly parametrize and train the neural network mapping. \tom{While the use of the reduced basis representation for the approximating map allows one to detect appropriate breadth for the neural network by avoiding complex neural network hyperparameter searches, and associated nonconvex neural network trainings, how to choose appropriate depth for the network is still an open question. While neural network approximation theory suggests deeper networks have richer representative capacities, in practice, for many architectures, adding depth eventually diminishes performance in what is known as the ``peaking phenomenon''\cite{Hughes1968}. In general finding appropriate depth for e.g., fully-connected feedforward neural networks requires re-training from scratch different networks with differing depths. In order to avoid this issue we employ an adaptively constructed residual network (ResNet) neural network parametrization of the mapping between the two reduced bases. This adaptive construction procedure is motivated by recent approximation theory that conceives of ResNets as discretizations of sequentially minimizing control flows \cite{LiLinShen22}, where such maps are proven to be universal approximators of $L^p$ functions on compact sets. A schematic for our neural network architecture can be seen in Fig.\ \ref{fig:dipresnet}.}

This strategy adaptively constructs and trains a sequence of low-rank ResNet layers, where for convenience we take $ r= r_M = r_F$ or otherwise employ a restriction or prolongation layer to enforce dimensional compatibility. The ResNet hidden neurons at layer $i+1$ have the form
\begin{equation}\label{eq:resnet}
   z_{i+1} = z_i + w_{i,2}\sigma(w_{i,1}z_i + b_i),
\end{equation} 
with $z_{i+1},z_i,b_i \in \mathbb{R}^{r}$, $w_{i,2}, w_{i,1}^T \in \mathbb{R}^{r\times k}$, where the parameter $k<r$ is referred to as the layer rank\tom{, and it is chosen to be smaller than $r$ in order to impose a compressed representation of the ResNet latent space update \eqref{eq:resnet}. This choice is guided by the ``well function'' property in \cite{LiLinShen22}}. The ResNet weights $\mathbf{w} = [(w_{i,2},w_{i,1},b_i)]_{i=0}^{\text{depth}}$ consist of all of the coefficient arrays in each layer. \tom{Given appropriate reduced bases with dimension $r$, the ResNet mapping between the reduced bases is trained adaptively, one layer at a time, until over-fitting is detected in training validation metrics. When this is the case, a final global end-to-end training is employed using a stochastic Newton optimizer \cite{OLearyRoseberryAlgerGhattas2020}.} This architectural strategy is able to achieve high generalizability for few and (input-output) dimension-independent data, \tom{for more information on this strategy, please see} \cite{OLearyRoseberryDuChaudhuriEtAl2021}.

\begin{figure}[!htb]
\begin{center}
    \includegraphics[width = 0.8\textwidth]{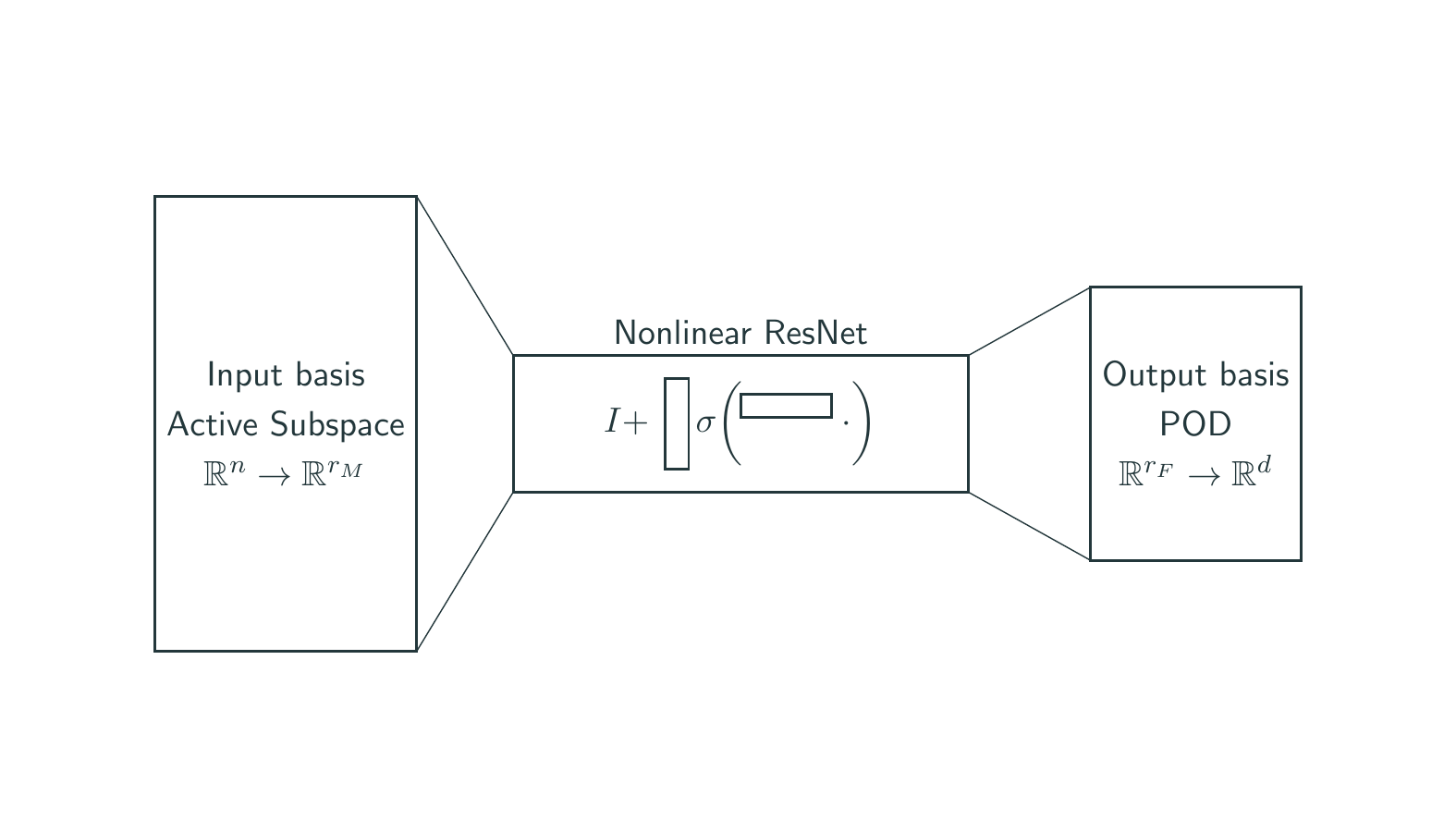}    
    \caption{A schematic representation of the derivative-informed projected neural network using ResNet as the nonlinear mapping between the reduced subspace using active subspaces for input parameters and POD for output observables.}\label{fig:dipresnet}
\end{center}
\end{figure}

By removing the dependence of the input-to-output map on their high-dimensional and uninformed subspaces (complement to the low-dimensional and informed subspaces), we can construct a neural network of small input and output size that requires few training data.
% and as a consequence, the dependence on large amounts of training data. 
Since these architectures are able to achieve high generalization accuracy with limited training data for parametric PDE maps, they are especially well suited to scalable EIG approximation, since they can be efficiently queried many times at no cost in PDE solves, and require few high fidelity PDE solutions for their construction.

%%================================%%
%% DLMC with DIPNet surrogate%%
%%================================%%

\subsection{DLMC with DIPNet surrogate} \label{sec:nndlmc}

We propose to train a DIPNet surrogate $\tilde{\bF}_d$ for the parameter-to-observable map ${\bF}_d$, so that $\hat{\pi}(\obs_i, \W)$ can be approximated as 
\begin{equation}\label{eq:tilde-pi}
    \tilde{\pi}(\obs_i, \W) = \frac{1}{n_{\text{in}}}\sum^{n_{\text{in}}}_{j=1}
    \exp\left(-\frac{1}{2}\|\obs_i-\W\tilde{\bF}_d({\bipar}_{i,j})\|^2_{\Gnoise^{-1}}\right),
\end{equation}
where we omitted a constant $\frac{1}{(2\pi)^{n_{\text{out}}/2}\det (\Gnoise)^{1/2}}$ since it appears in both the numerator and denominator of the argument of the log in the expression for the EIG. To this end, we can formulate the approximate EIG with the DIPNet surrogate as 
\begin{align}\label{eq:Psi-nn}
    \Psi^{nn} &:= \frac{1}{n_{\text{out}}} \sum^{n_{\text{out}}}_{i=1}\log\left(\frac{\pilike(\obs_i\vert\bipar_i, \W)}{\tilde{\pi}(\obs_i, \W)}\right)\\
    & = - \frac{1}{n_{\text{out}}} \sum^{n_{\text{out}}}_{i=1} \frac{1}{2}\|\W \obsn_{d,i}\|^2_{\Gnoise^{-1}} \\
    & \quad - \frac{1}{n_{\text{out}}} \sum^{n_{\text{out}}}_{i=1} \log \left(\frac{1}{n_{\text{in}}}\sum^{n_{\text{in}}}_{j=1}\exp\left(-\frac{1}{2}\|\obs_i-\W\tilde{\bF}_d({\bipar}_{i,j})\|^2_{\Gnoise^{-1}}\right) \right),
    \label{eq:nneig}
\end{align}
where $\obsn_{d,i}$ are i.i.d.\ observation noise. Thanks to the DIPNet surrogate, the EIG can be evaluated at negligible cost (relating to PDE solver cost) for each given $\W$, and does not require any PDE solves. 

%%================================%%
%% Error analysis%%
%%================================%%

\subsection{Error analysis} \label{sec:error}
% For simplicity, we ignore the constant term $\frac{1}{(2\pi)^{n_{\text{out}}/2}\det (\Gnoise)^{1/2}}$ in $ \tilde{\pi}$ and $\hat{\pi}$ as it gets cancelled as in (\ref{eq:nneig}).

\begin{theorem}\label{thm:error}
We assume that the parameter-to-observable map $\bF_d$ and its surrogate $\tilde{\bF}_d$ are bounded as 
\begin{equation}\label{eq:bound-E}
    \mathbb{E}_{\bipar \sim \mu_{\text{pr}}} [\| \bF_d(\bipar)\|^2] < \infty \quad \text{and} \quad \mathbb{E}_{\bipar \sim \mu_{\text{pr}}} [\| \tilde{\bF}_d(\bipar)\|^2] < \infty.
\end{equation}
Moreover, we assume that
the generalization error of the DIPNet surrogate can be bounded by $\varepsilon$, i.e., 
\begin{equation}\label{eq:bound-error}
    \mathbb{E}_{\bipar \sim \mu_{\text{pr}}} [\| \bF_d(\bipar) - \tilde{\bF}_d(\bipar)\|^2] \leq \varepsilon^2.
\end{equation}
Then the error in the approximation of the normalization constant by the DIPNet surrogate can be bounded by 
\begin{equation}\label{eq:pi-error}
    \vert \hat{\pi} (\obs_i, \W) - \tilde{\pi}(\obs_i, \W)\vert \leq C_i \varepsilon,
\end{equation}
for sufficiently large $n_{\text{in}}$ and some constants $0 < C_i < \infty$, $i = 1, \dots, n_{\text{out}}$. Moreover, 
the approximation error for the EIG can be bounded by 
\begin{equation}\label{eq:Psi-error}
    \vert\Psi^{dl}(\W) - {\Psi}^{nn}(\W)\vert \leq C \varepsilon
\end{equation}
for some constant $0 < C < \infty$.
\end{theorem}
\begin{proof}
    For notational simplicity, we omit the dependence of $\hat{\pi}$  and $\tilde{\pi}$ on $\W$ and write $\bF = \W \bF_d$ and $\tilde{\bF} = \W \tilde{\bF}_d$. Note that the bounds \eqref{eq:bound-E} and \eqref{eq:bound-error} also hold for $\bF$ and $\tilde{\bF}$ since $\bF$ and $\tilde{\bF}$ are selection of some entries of $\bF_d$ and $\tilde{\bF}_d$. 
    By definition of $\hat{\pi}$ in \eqref{eq:hat-pi} and $\tilde{\pi}$ in \eqref{eq:tilde-pi}, and the fact that $\vert e^{-x} - e^{-x'}\vert \leq \vert x - x'\vert$ for any $x, x' > 0$, we have 
    \begin{align*}
        & \vert \hat{\pi}(\obs_i) - \tilde{\pi}(\obs_i) \vert, \\
        & \leq  \frac{1}{n_{\text{in}}}\sum^{n_{\text{in}}}_{j=1} \frac{1}{2} \bigg\lvert \|\obs_i - \bF(\bipar_{i,j})\|^2_{\Gnoise^{-1}} - \|\obs_i - \tilde{\bF}(\bipar_{i,j})\|^2_{\Gnoise^{-1}} \bigg\rvert, \\ 
        & =  \frac{1}{n_{\text{in}}}\sum^{n_{\text{in}}}_{j=1} \frac{1}{2} \bigg\lvert \left(2\obs_i - \bF(\bipar_{i,j}) - \tilde{\bF}(\bipar_{i,j})\right)^T \Gnoise^{-1} \left(\bF(\bipar_{i,j}) - \tilde{\bF}(\bipar_{i,j})\right)\bigg\rvert, \\
        & \leq \frac{1}{n_{\text{in}}}\sum^{n_{\text{in}}}_{j=1} \| \Gnoise^{-1}\| (2\|\obs_i\|+ \|\bF(\bipar_{i,j})\| + \|\tilde{\bF}(\bipar_{i,j})\|)\|\bF(\bipar_{i,j}) - \tilde{\bF}(\bipar_{i,j})\|,\\
        & \leq C_i \left(\frac{1}{n_{\text{in}}}\sum^{n_{\text{in}}}_{j=1}  \|\bF(\bipar_{i,j}) - \tilde{\bF}(\bipar_{i,j})\|^2 \right)^{1/2} ,
    \end{align*}
    where we used the Cauchy--Schwarz inequality in the last inequality with $C_i$ given by 
    \begin{equation}
        C_i = \| \Gnoise^{-1}\| \left(4 \|\obs_i\|^2 + 
        \left( \frac{1}{n_{\text{in}}}\sum^{n_{\text{in}}}_{j=1}  \|\bF(\bipar_{i,j})\|^2 \right)^{1/2} + 
        \left( \frac{1}{n_{\text{in}}}\sum^{n_{\text{in}}}_{j=1} \|\tilde{\bF}(\bipar_{i,j})\|^2 \right)^{1/2}
        \right).
    \end{equation}
    For sufficiently large $n_{\text{in}}$, we have that $C_i < \infty$ by $\|\obs_i\| < \infty$ and the assumption \eqref{eq:bound-E}. Moreover, the error bound \eqref{eq:pi-error} holds by the assumption \eqref{eq:bound-error}. 

    By the definition of the EIGs \eqref{eq:Psi-dl} and \eqref{eq:Psi-nn}, we have 
    \begin{equation}
        \vert\Psi^{dl}(\W) - {\Psi}^{nn}(\W)\vert \leq \frac{1}{n_{\text{out}}} \sum^{n_{\text{out}}}_{i=1} \lvert \log(\hat{\pi}(\obs_i, \W)) - \log(\tilde{\pi}(\obs_i, \W))\rvert.
    \end{equation}
    For sufficiently large $n_{\text{in}}$, we have that the normalization constants $\hat{\pi}(\obs_i, \W)$ and $\tilde{\pi}(\obs_i, \W)$ are bounded away from zero, i.e., 
    \begin{equation}
        \hat{\pi}(\obs_i, \W), \tilde{\pi}(\obs_i, \W) \geq c_i,
    \end{equation}
    for some constant $c_i > 0$. Then we have 
    \begin{equation}
        \vert\Psi^{dl}(\W) - {\Psi}^{nn}(\W)\vert \leq \frac{1}{n_{\text{out}}} \sum^{n_{\text{out}}}_{i=1} \frac{1}{c_i} \vert \hat{\pi} (\obs_i, \W) - \tilde{\pi}(\obs_i, \W)\vert \leq \frac{1}{n_{\text{out}}} \sum^{n_{\text{out}}}_{i=1} \frac{C_i}{c_i} \varepsilon,
    \end{equation}
    where we used the bound \eqref{eq:pi-error}, which implies the bound \eqref{eq:Psi-error} with constant 
    \begin{equation}
        C = \frac{1}{n_{\text{out}}} \sum^{n_{\text{out}}}_{i=1} \frac{C_i}{c_i}.
    \end{equation}
\end{proof}

\subsection{Greedy algorithm} \label{sec:greedy}
With the DIPNet surrogate, the evaluation of the DLMC EIG $\Psi^{nn}$ defined in (\ref{eq:nneig}) does not involve any PDE solves. Thus to solve the optimization problem
\begin{equation}
\label{eq:nnop}
   \W^* = \argmax_{\W \in \mathcal{W}} \Psi^{nn}(\W),
\end{equation}
we can directly use a greedy algorithm that requires evaluations of
the EIG, not its derivative w.r.t. $\W$. Let $S_d$ denote the set of
all $d$ candidate sensors; we wish to choose $r$ sensors from $S_d$
that maximize the approximate EIG (approximated with the DIPNet
surrogate). At the first step with $t = 1$, we select the sensor $v^*
\in S_d$ corresponding to the maximum approximate EIG and set $S^* =
\{v^*\}$. Then at step $t = 2, \dots, r$, with $t-1$ sensors selected
in $S^*$, we choose the $t$-th sensor $v^* \in S_d \setminus S^*$
corresponding to the maximum approximate EIG evaluated with $t$
sensors $S^* \cup \{v^*\}$; see Algorithm \ref{alg:greedy} for the
greedy optimization process, which can achieve (quasi-optimal)
experimental designs with an approximation guarantee under suitable
assumptions on the incremental information gain of an additional
sensor; see \cite{JagalurMohanMarzouk21} and references therein. Note
that at each step the approximate EIG can be evaluated in parallel for
each sensor choice $S^* \cup \{v\}$ with $v \in S_d \setminus S^*$. 
% More in detail,
% The greedy algorithm sequentially finds each optimal design one by one that maximize the current EIG with detailed description in Algorithm \ref{alg:greedy}. 
% In each iteration, we find the next optimal sensor to add to the current optimal sensor set $S^*$. Assume we obtain the first $t-1$ optimal sensors in $S^{*}$, all the possible $t$ sensor designs are $S^{*} \cup \{ v\}$ with $v \in S_d \setminus S^{*}$. Then we can evaluate $\Psi^{nn}$ for each design in parallel to find the design choice that has the maximum value of $\Psi^{nn}$ as the current $t$ optimal sensor set $S^{*}$. We repeat this process until we have $r$ sensors in $S^*$.

\begin{algorithm}
\caption{Greedy algorithm to solve (\ref{eq:nnop})}\label{alg:greedy}
\begin{algorithmic}[1]
\Require data $\{ \obs_i\}^{N_s}_{i=1}$ generated from the prior samples $ \{ \bs{m}_i\}^{N_s}_{i=1}$, $d$ sensor candidates set $S_d$, sensor budget $r$, optimal sensor set $S^* = \emptyset$
\Ensure optimal sensor set $S^*$
\For{$t = 1,\dots,r$}
	\State $S \Leftarrow S_d \setminus S^*$
	\For{$v \in S$}
		\State $\W_v$ is the design matrix of sensor choice $S^*\cup \{ v\}$
		\State evaluate $\Psi^{nn}(\W_v)$
	\EndFor
	\State $v^* \Leftarrow \argmax_{v \in S} \Psi^{nn}(\W_v)$
	\State $S^* \Leftarrow S^* \cup \{ v^*\}$
\EndFor
\end{algorithmic}
\end{algorithm}

%%================================%%
%% Numerical results%%
%%================================%%
\section{Numerical results} \label{sec:numerical}

In this section, we present numerical results for OED problems involving a Helmholtz acoustic inverse scattering problem and an advection-reaction-diffusion inverse transport problem to illustrate the efficiency and accuracy of our method. We compare the approximated normalization constant and EIG of our method with 1) the DLMC truth computed by a large number of Monte Carlo samples; and 2) the DLMC sampled at the same computational cost (number of PDE solves) as our method using DIPNet training.

Both PDE problems are discretized using the finite element library \texttt{FEniCS} \cite{AlnaesBlechtaHakeEtAl2015}. The construction of training data and reduced bases (active subspace and proper orthogonal decomposition) is implemented in \texttt{hIPPYflow} \cite{hippyflow}, a library for dimension reduced PDE surrogate construction, building on top of PDE adjoints implemented in \texttt{hIPPYlib} \cite{VillaPetraGhattas20}. The DIPNet neural network surrogates are constructed in \texttt{TensorFlow} \cite{AbadiBarhanChenEtAl2016}, and are adaptively trained using a combination of Adam \cite{KingmaBa2014}, and a Newton method, LRSFN, which improves local convergence and generalization \cite{OLearyRoseberryAlgerGhattas2019,OLearyRoseberryAlgerGhattas2020,OLearyRoseberryDuChaudhuriEtAl2021}.

%%================================%%
%% Helmholtz problem%%
%%================================%%

\subsection{Helmholtz problem} \label{sec:Helmholtz}

For the first numerical experiment, we consider 
% an optimal experimental design problem where the underlying physics are 
an inverse wave scattering problem modeled by the Helmholtz equation with uncertain medium in the two-dimensional physical domain $\Omega = (0,3)^2$
\begin{subequations}
\begin{align}
  - \Delta {u}  - e^{2m} k^2 u &= f \quad \text{in } \Omega, \\
   \text{PML boundary condition} &\text{ on } \partial \Omega \setminus \Gamma_\text{top}, \\
   \nabla u \cdot \mathbf{n} &= 0  \text{ on }  \Gamma_\text{top}, \\
  \mathcal{F}_d(m) &= [u(\bs{x}_i,m)] \quad \text{at } \bs{x}_i \in \Omega.
\end{align}
\end{subequations}
where $u$ is the total wave field, $k \approx 4.55$ is the wavenumber, and $e^{2m}$ models the uncertainty of the medium, with the parameter $\ipar$ a log-prefactor of the squared wavenumber. The right hand side $f$ is a point source located at $\mathbf{x} = (0.775,2.5)$.
% is a log-prefactor of a variable coefficient wavenumber in two dimensions. 
% The formulation of the PDE parameter to candidate sensors problem is
The perfectly matched layer (PML) boundary condition approximates a semi-infinite domain. The candidate sensor locations $\bs{x}_i$ are linearly spaced in the line segment between the edge points $(0.1,2.9)$ and $(2.9,2.9)$, \kw{with coordinates $\{ (0.1+2i/35, 2.9) \}_{i=0}^{49}$}, as shown in Fig.\ \ref{fig:h0}. The prior distribution for the uncertain parameter $m$ is Gaussian $\mu_\text{pr} = \mathcal{N}(\mpr,\Cpr)$ with zero mean $\mpr = 0$ and covariance $\Cpr = (5.0I - \Delta)^{-2}$. The mesh used for this problem is uniform of size $128 \times 128$. 
% The first order continuous Galerkin elements are used for the parameter $\mathbf{m}$, and second order continuous Galerkin elements are used for the vector wavefield $u$. 
We use quadratic elements for the discretization of the wave field $u$ and linear elements for the parameter $m$, leading to a discrete parameter $\bipar$ of dimension $16,641$. \tom{The dimension of the wave field $u$ is $66049$, the wave is more than sufficiently resolved in regards to the Nyquist sampling criteria for wave problems.} 
 % All the possible observation sensor locations are near the upper boundary of the domain. 
 A sample of the parameter field $m$ and the PDE solution $u$ is shown in Fig.\ \ref{fig:h0} with all $50$ candidate sensor locations marked in circles.

The network has $10$ low-rank residual layers, each with a layer rank of $10$. \tom{For this numerical example we demonstrate the effects of using different breadths in the neural network representation, in each case the ResNet learns a mapping from the first $r_M$ basis vectors for active subspace to the first $r_M$ basis vectors of POD. In the case that $r_M>50$ we use a linear restriction layer to reduce the ResNet latent representation to the $50$ dimensional output. For the majority of the numerical results, we employ a $r_M = 50$ dimensional network. 
The neural network is trained adaptively using $4915$ training samples, and $1228$ validation samples. Using $512$ independent testing samples, the DIPNet surrogate was $81.56\%$ $\ell^2$ accurate measured as a percentage by
\begin{equation}\label{eq:ell2acc}
  \ell^2 \text{ accuacy} = 100\left(1 - \frac{\|\tilde{\mathbf{F}}_d - \mathbf{F}_d\|_2}{\|\mathbf{F}_d\|_2}\right).
\end{equation}
For more details on this neural network architecture and training, see \cite{OLearyRoseberryDuChaudhuriEtAl2021}. The computational cost of the $50$ dimensional active subspace projector using $128$ samples is equivalent to the cost of $842$ additional training data; since the problem is linear the additional linear adjoint computations are comparable to the costs of the training data generation. As we will see in the next example, when the PDE is nonlinear the additional active subspace computations are marginal. Thus for fair comparison with the same computational cost of PDE solves, we use $4915 + 842 = 5757$ samples for simple MC.}

\begin{figure}[ht]
% \vskip 0.2in
\begin{center}
\centerline{\includegraphics[width=.45\columnwidth]{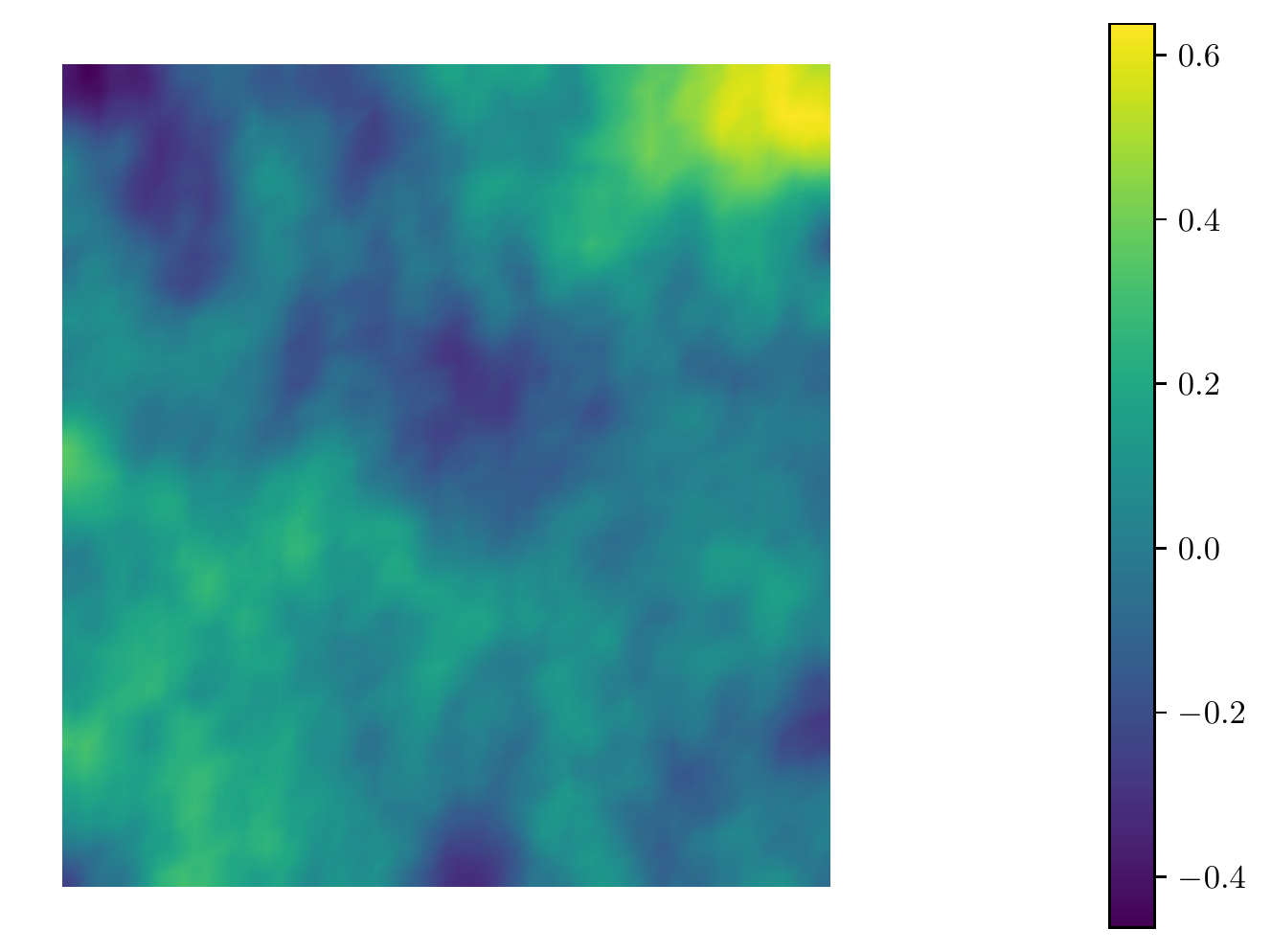}\includegraphics[width=.45\columnwidth]{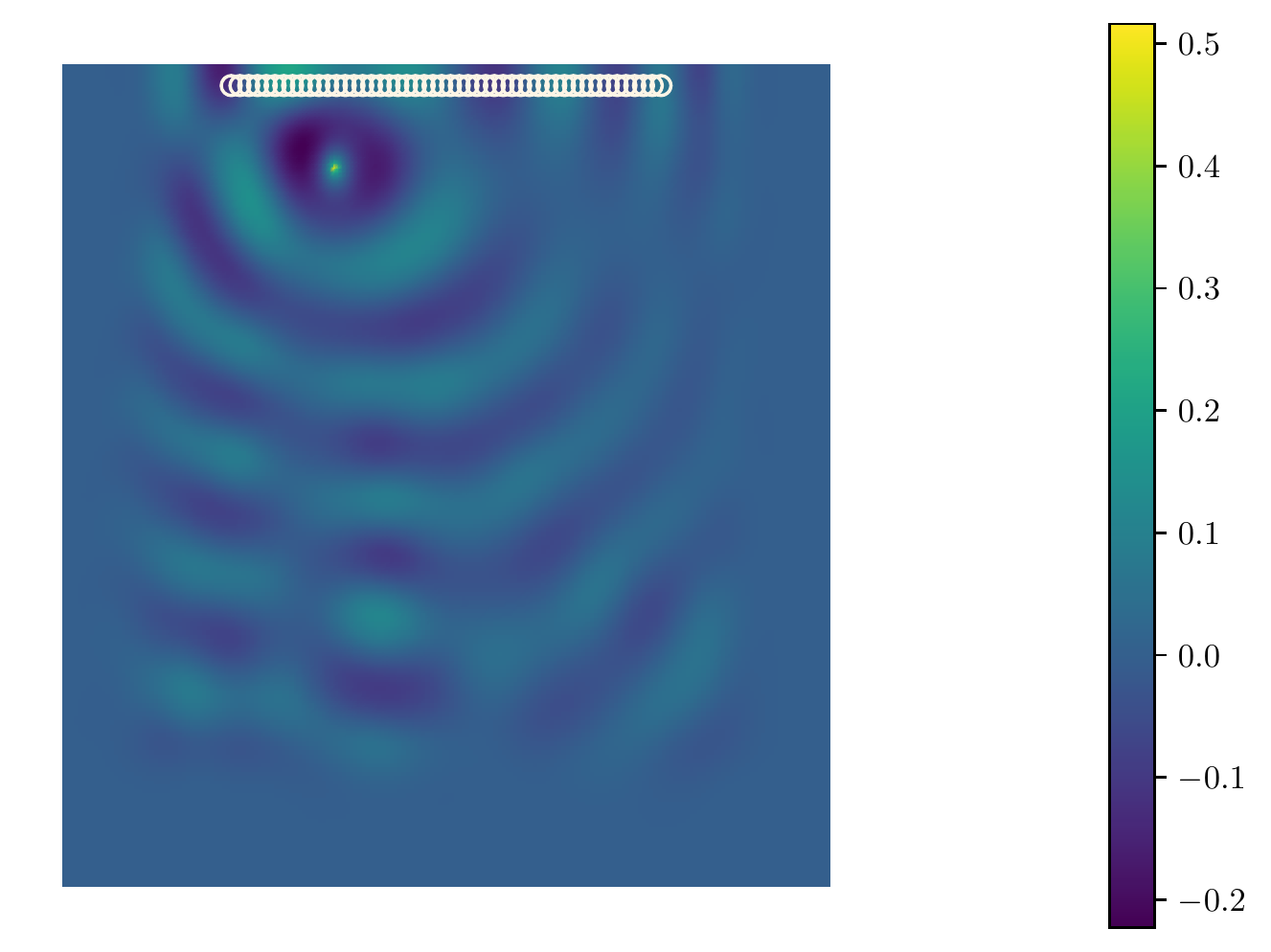}}
\caption{A random sample of the parameter field $\bipar$ (left) and the corresponding solution $\mathbf{u}$ with candidate observation sensor locations marked in circles (right) for the Helmholtz problem.}
\label{fig:h0}
\end{center}
\vskip -0.2in
\end{figure}
To test the efficiency and accuracy of our DIPNet surrogate, we first compute the log normalization constant $\log\pi(\obs)$ with our DIPNet surrogate for given observation data $\obs$ generated by $\obs = \W\bF_d(\bipar) +\obsn$. $\bipar$ is a random sample from the prior. We use in total $60000$ random samples for Monte Carlo (MC) to compute the normalization constant as the ground truth.  Fig.\ \ref{fig:h1} shows the $\log\pi(\obs)$ comparison for three random designs $\W$ that choose $15$ sensors out of $50$ candidates. We can see that DIPNet surrogate converges to a value close to the ground truth MC reference, 
% in a similar convergence rate. , this is not obvious,
while for the (simple) MC with $5757$ samples, the approximated value indicated by the green star is much less accurate than the DIPNet surrogate using $60000$ samples with similar cost in PDE solves. Note that the DIPNet training and evaluation cost for this small size neural network is negligible compared to PDE solves. 
% , the approximation is obviously way off as the green star.  

\begin{figure}[ht]
\vskip 0.2in
\begin{center}
\centerline{\includegraphics[width=.33\columnwidth]{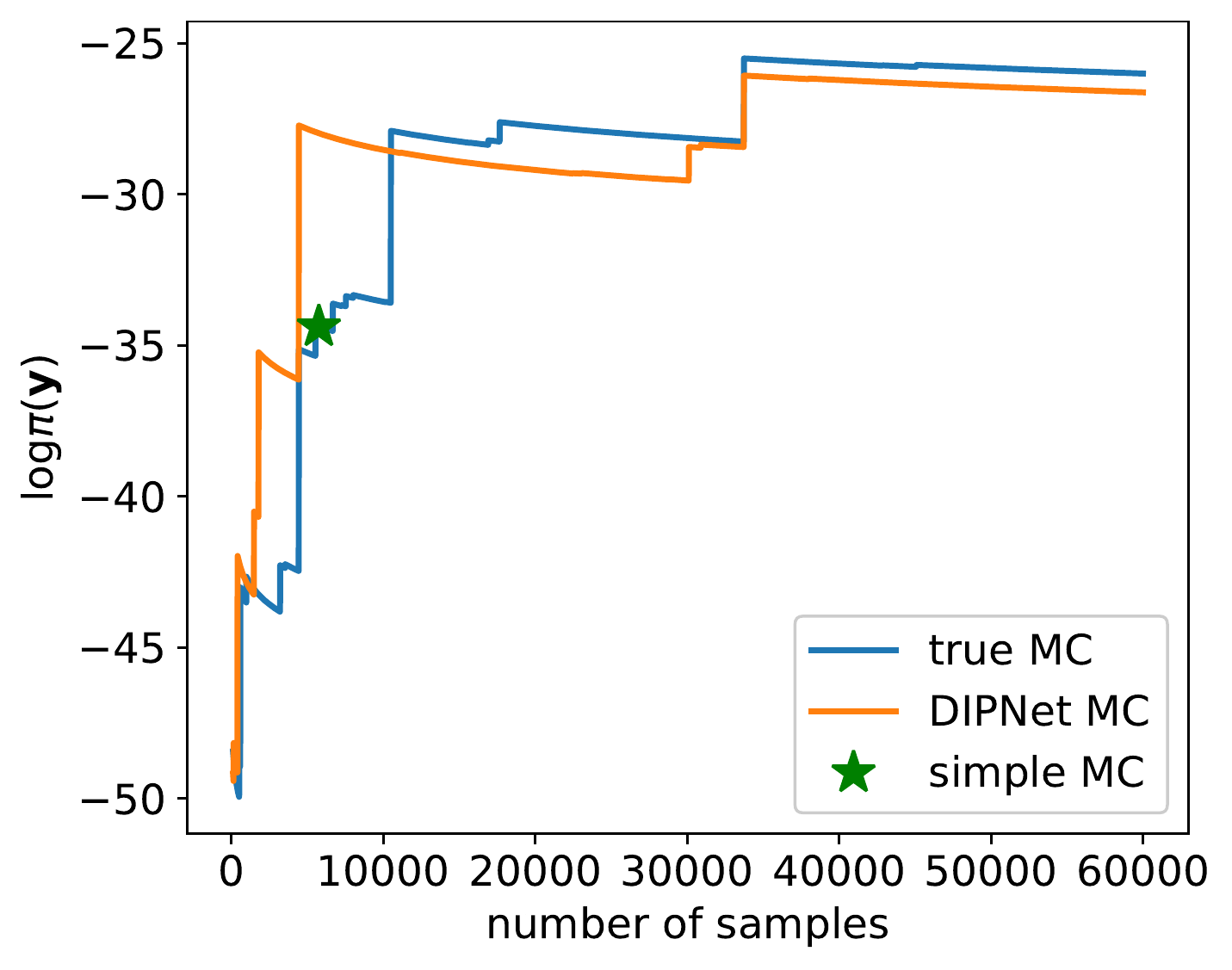}\includegraphics[width=.33\columnwidth]{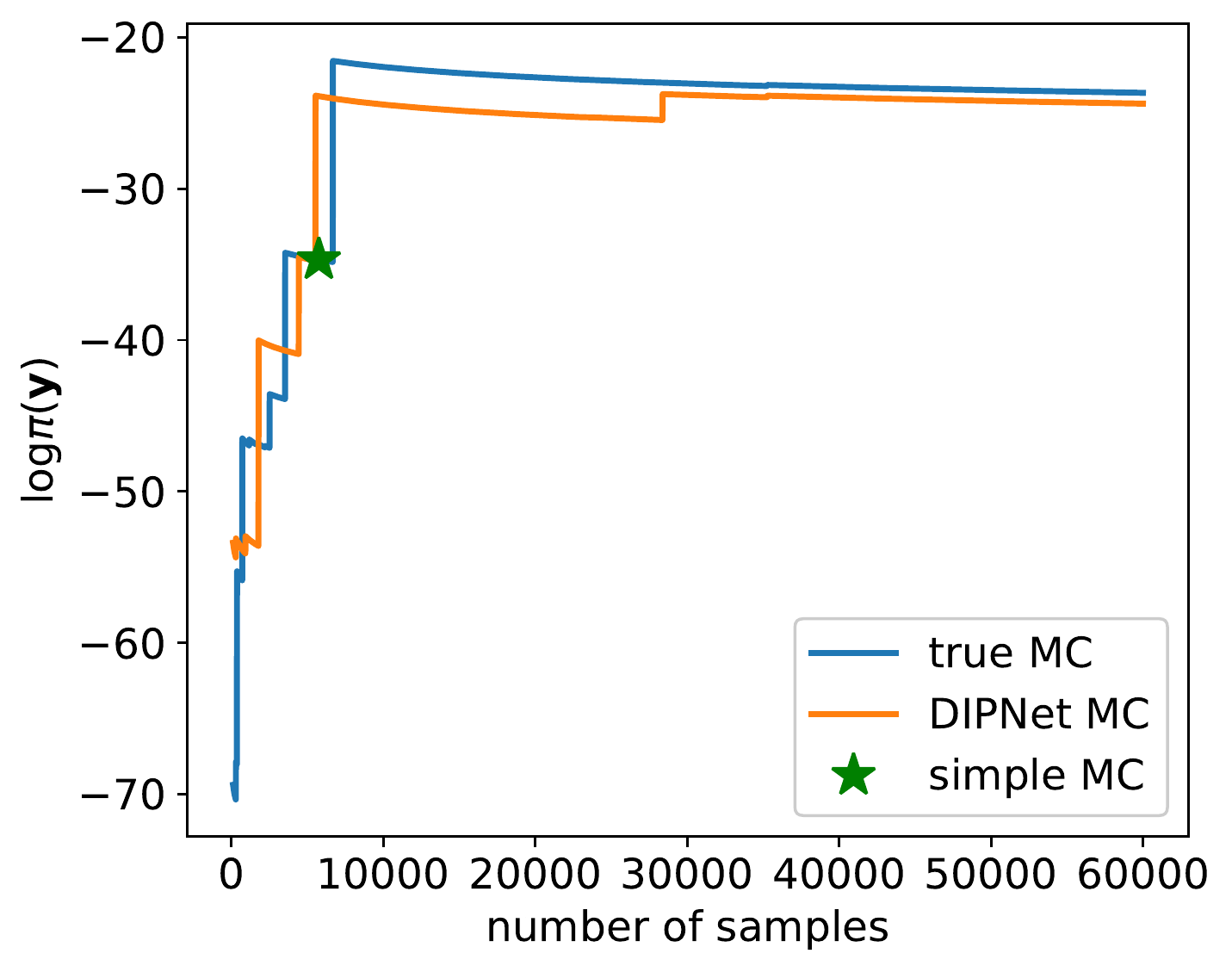}\includegraphics[width=.33\columnwidth]{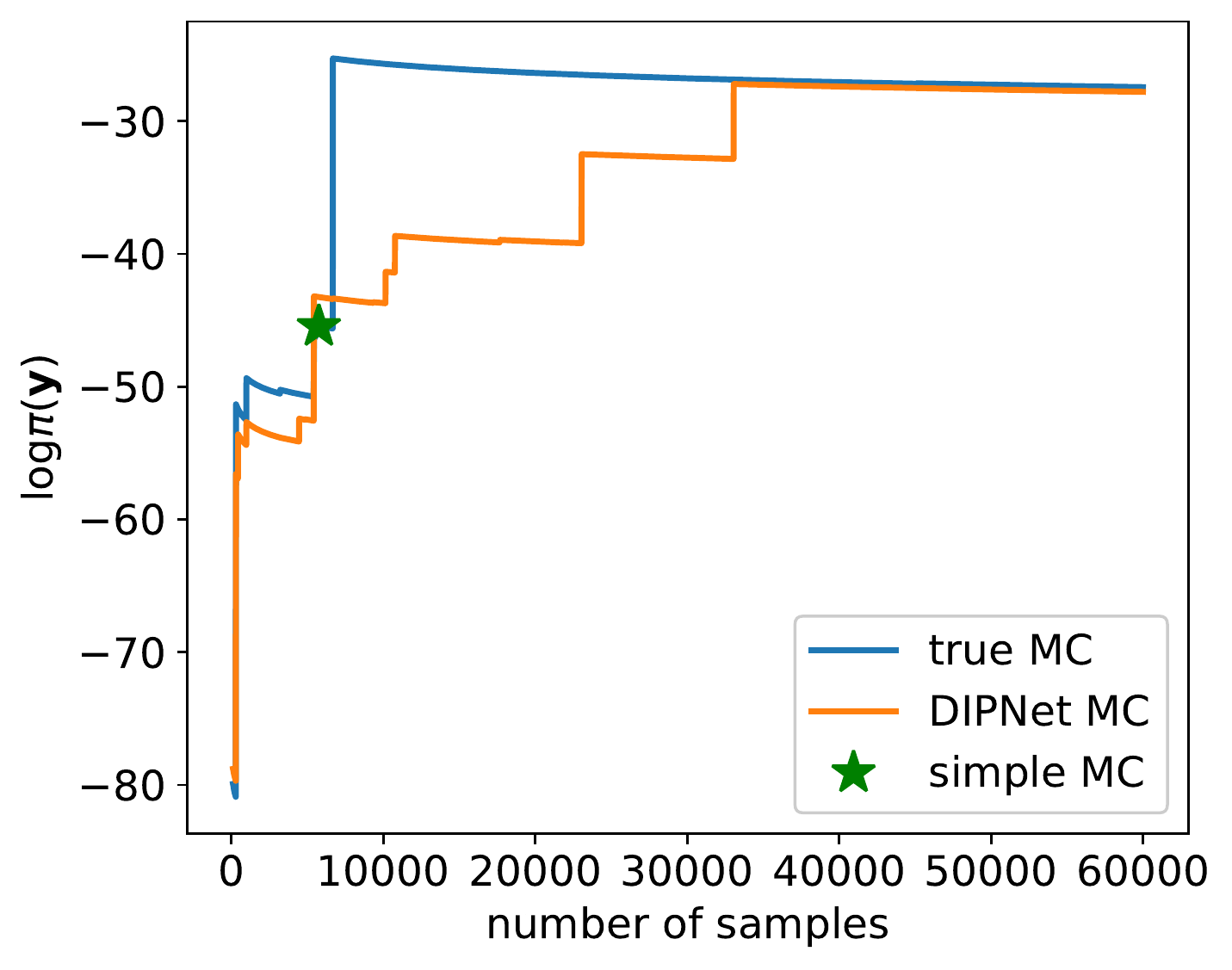}}
\caption{The approximation of the log normalization constant with increasing numbers of samples without (true MC) and with (DIPNet MC) surrogate at $3$ random designs. Green stars indicate $5757$ samples (for simple MC), which is the same computational cost as DIPNet.}
\label{fig:h1}
\end{center}
\vskip -0.2in
\end{figure}

The left and middle figures in Fig.\ \ref{fig:h2} illustrate the sample distributions of the relative approximation errors for the log normalization constant $\log\pi(\obs)$ and the EIG $\Psi$ ($n_{\text{out}} = 200$) with (DIPNet MC with $60000$ samples) and without (simple MC with $5757$ samples) the DIPNet surrogate, compared to the true MC with $60000$ samples. These results show that using DIPNet surrogate we can approximate $\log\pi(\obs)$ and $\Psi$ much more accurately with less bias compared to the simple MC. The sample distributions of EIG at $200$ random designs compared to the design chosen by the greedy optimization using DIPNet surrogate for different number of sensors are shown in the right figure, from which we can see that our method can always chose better designs with larger EIG values than all the random designs.

\begin{figure}[ht]
% \vskip 0.2in
\begin{center}
\centerline{\includegraphics[width=.337\columnwidth]{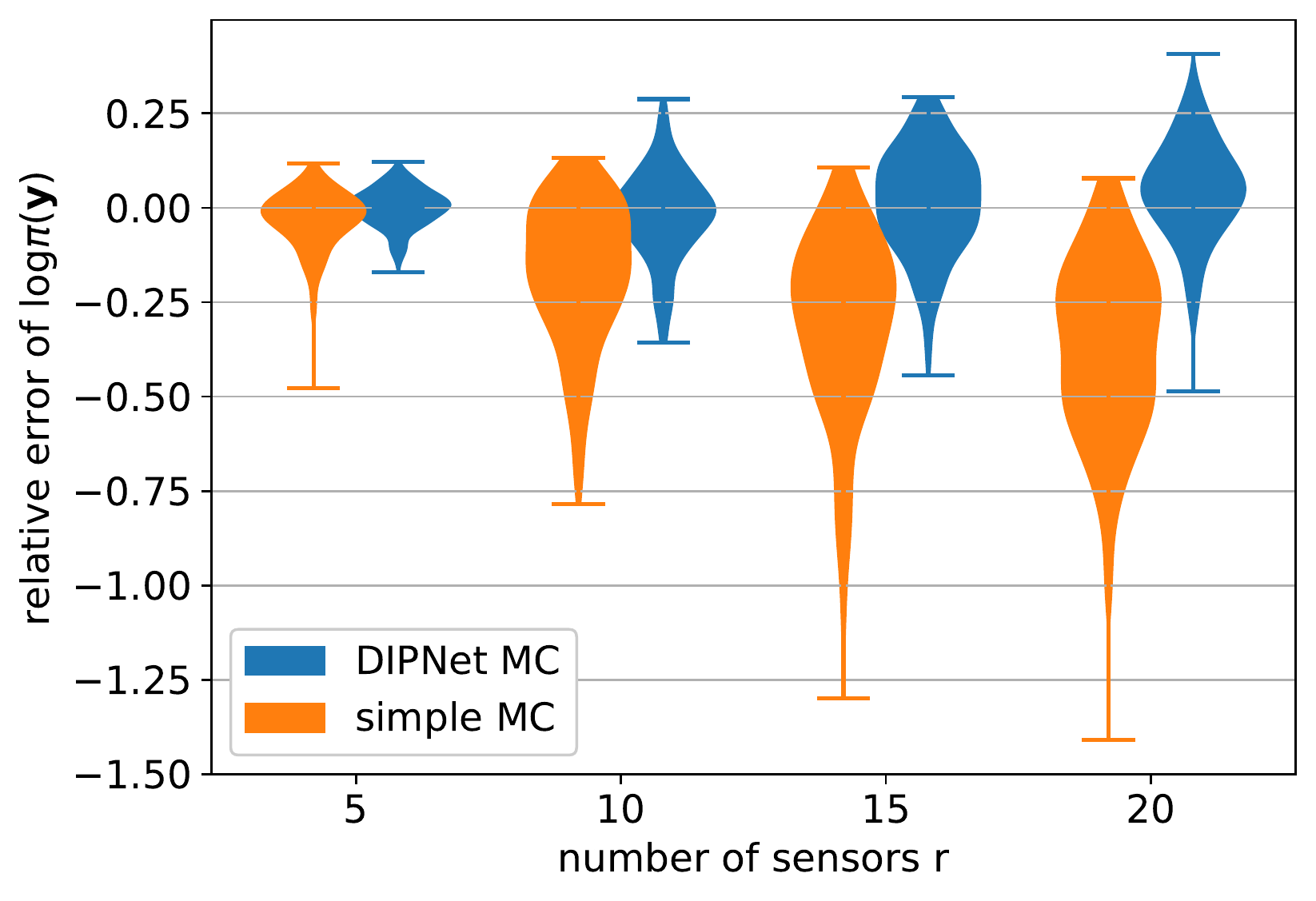}\includegraphics[width=.33\columnwidth]{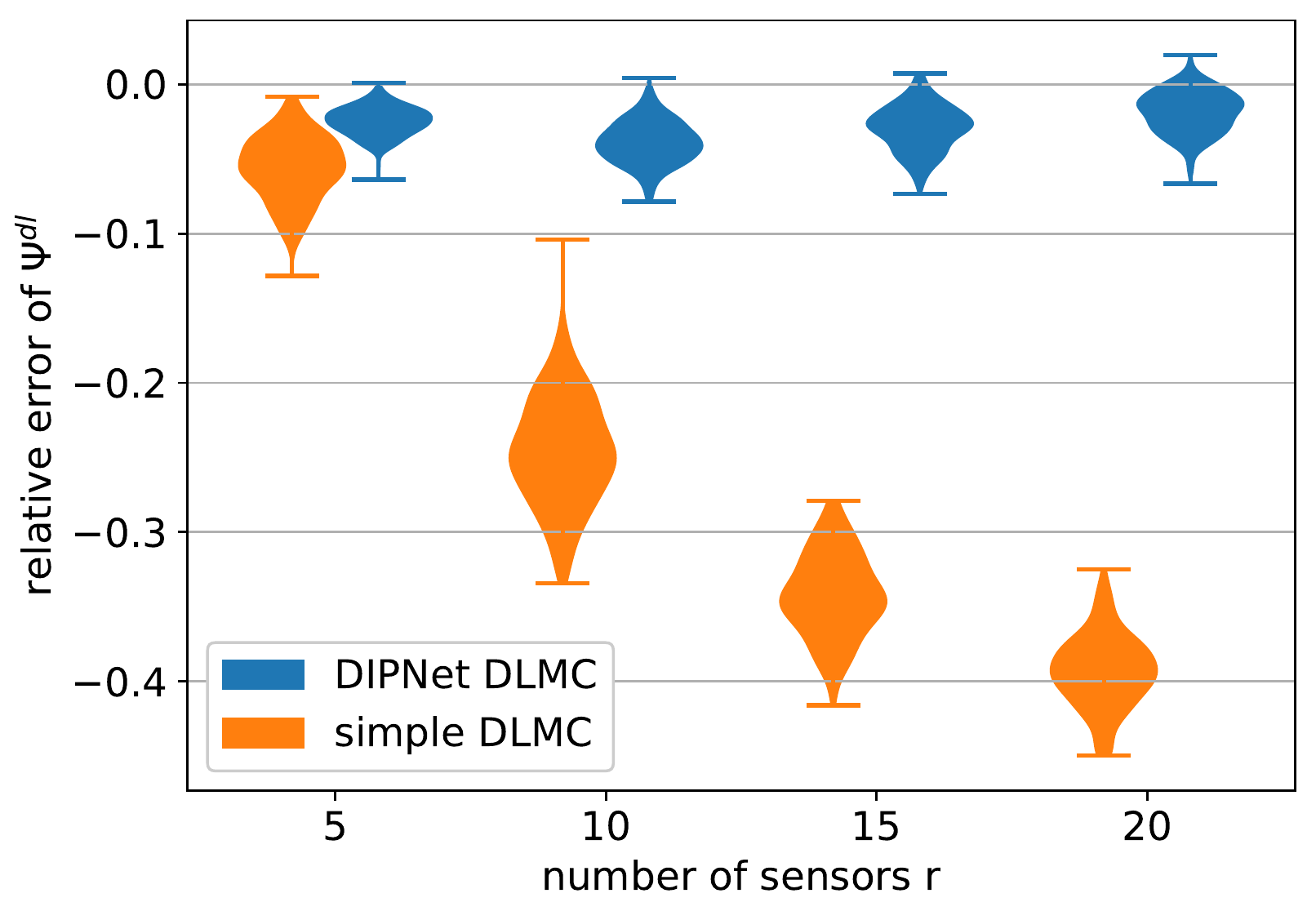}\includegraphics[width=.321\columnwidth]{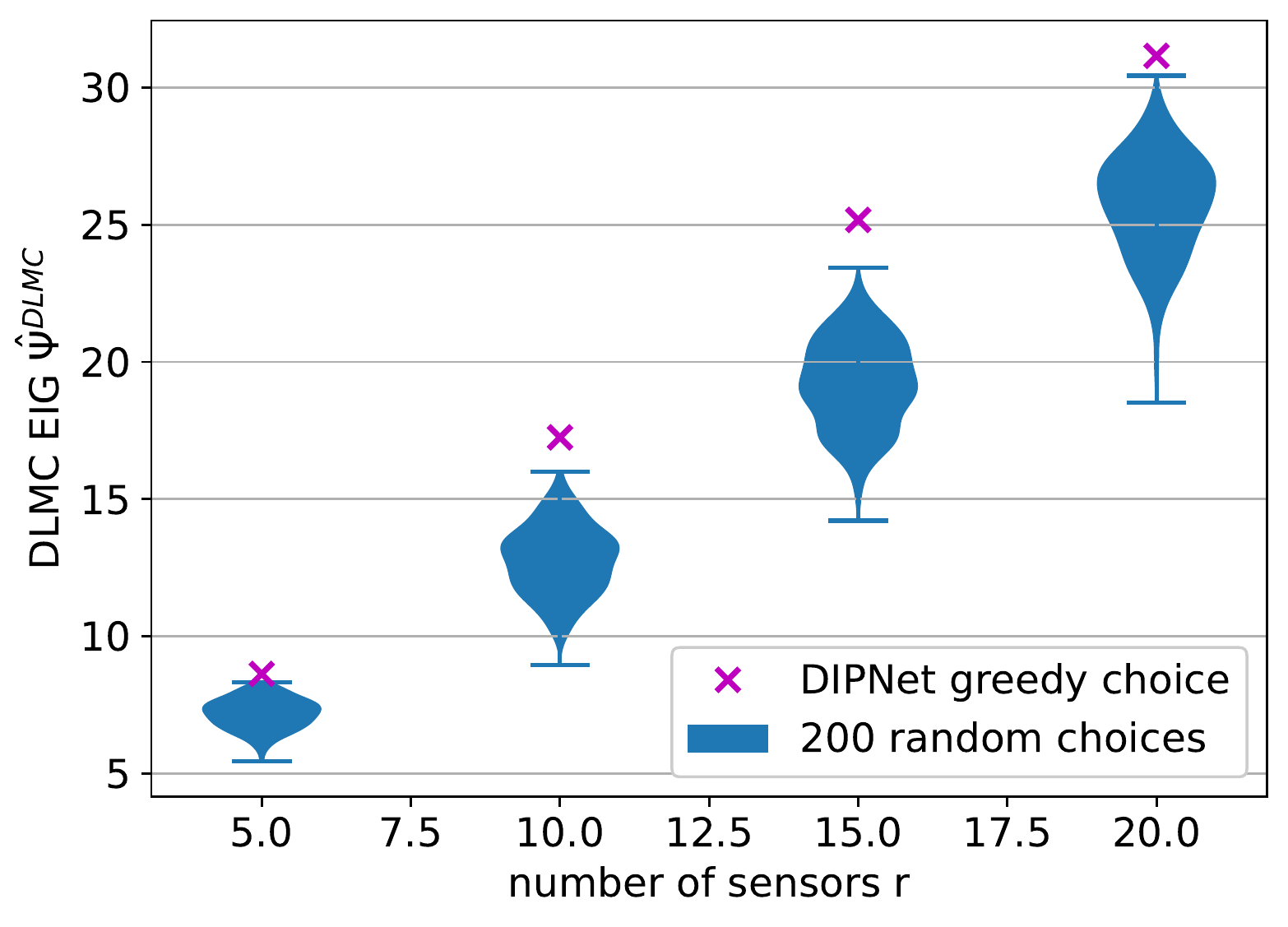}}
\caption{Sample distributions for 200 random designs of the relative approximation errors (compared to true MC) for the log normalization constant $\log\pi(\obs)$ (left) and EIG $\Psi$ (middle) by DIPNet MC and simple MC with different number of sensors $r$.
% Middle: sample distribution of the relative approximation error of DLMC EIG $\Psi^{dl}$ for DIPNet MC and MC with the same number of PDE solves against the true MC; 
Right: Blue filled areas represent the sample distributions of the true DLMC EIG $\Psi^{dl}$ for $200$ random designs. Pink crosses is the true DLMC EIG $\Psi^{dl}$ of designs chosen by the greedy optimization using DIPNet surrogates.}
\label{fig:h2}
\end{center}
\vskip -0.2in
\end{figure}

% Now we can approximate the DLMC EIG $\Psi^{dl}$ using the DIPNet surrogate. 
Fig.\ \ref{fig:h3} gives approximate values of the EIG with increasing number of outer loop samples $n_{\text{out}}$ using: the DIPNet surrogate with inner loop $n_{\text{in}} = 60000$, simple DLMC with inner loop $n_{\text{in}} = 5757$, and the truth computed with inner loop $n_{\text{in}} = 60000$. We can see that the approximate values by the DIPNet surrogates are almost the same as the truth, while simple DLMC results are very inaccurate. 
% We can further see it in the distribution of the relative error of $200$ random samples in the middle figure of Fig.\ \ref{fig:h2}. The DIPNet has a much smaller error mean and variance than the simple MC.
% Finally, in Fig.\ \ref{fig:h3}, we can see that the optimal designs chosen by the greedy algorithm with the DIPNet surrogate DLMC always have larger EIG than all the random designs. 

\begin{figure}[ht]
% \vskip 0.2in
\begin{center}
\centerline{\includegraphics[width=.33\columnwidth]{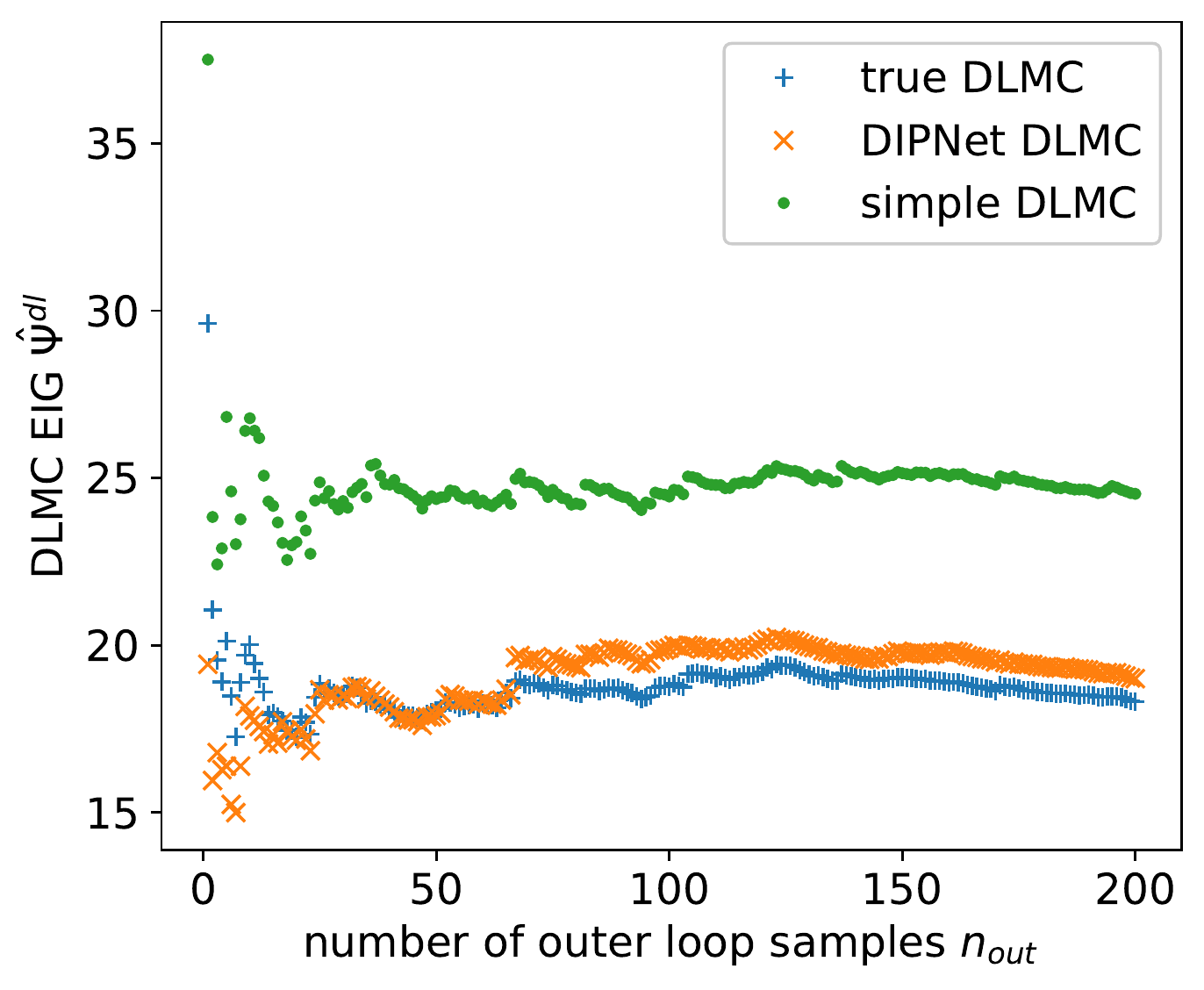}\includegraphics[width=.33\columnwidth]{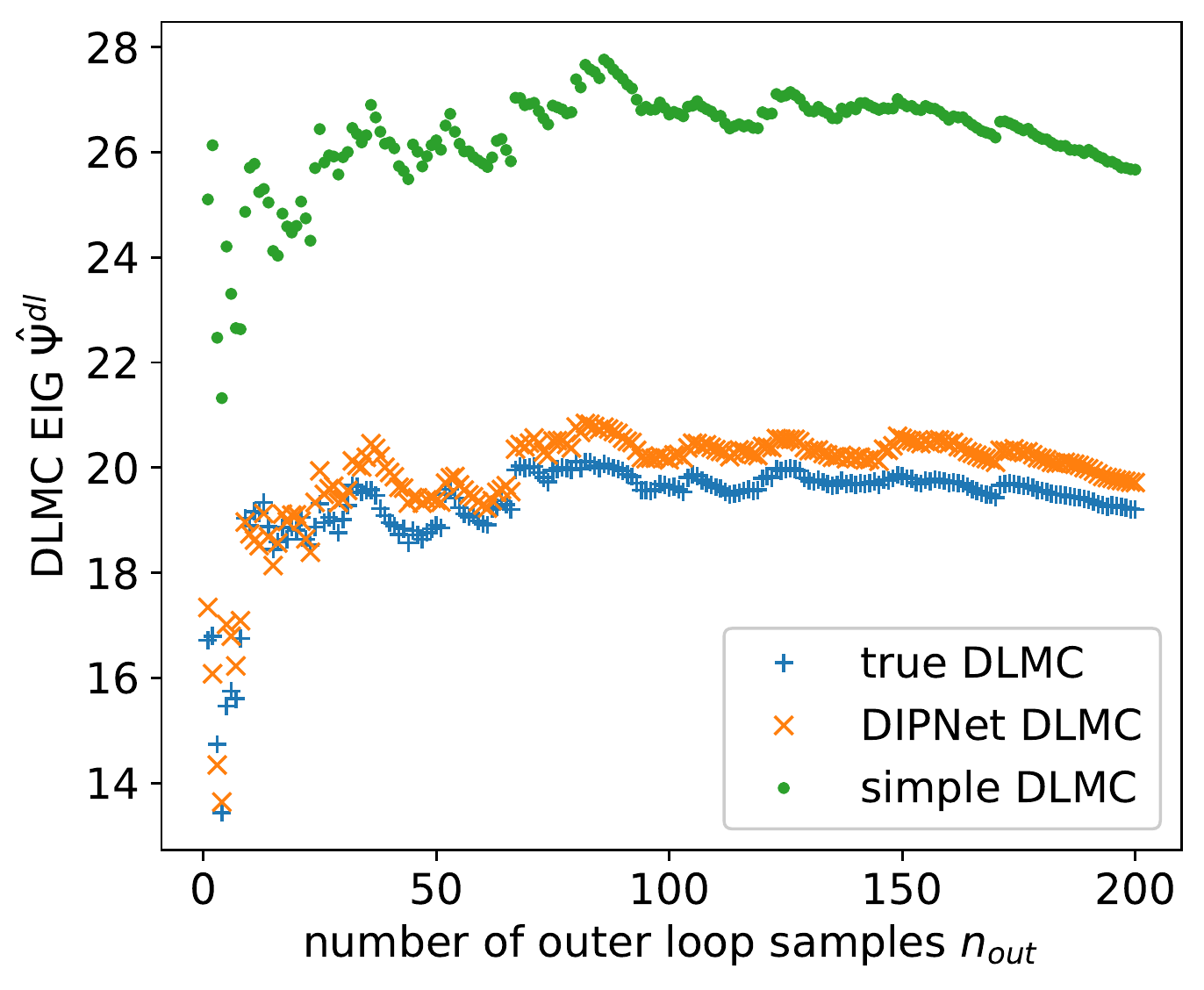}\includegraphics[width=.337\columnwidth]{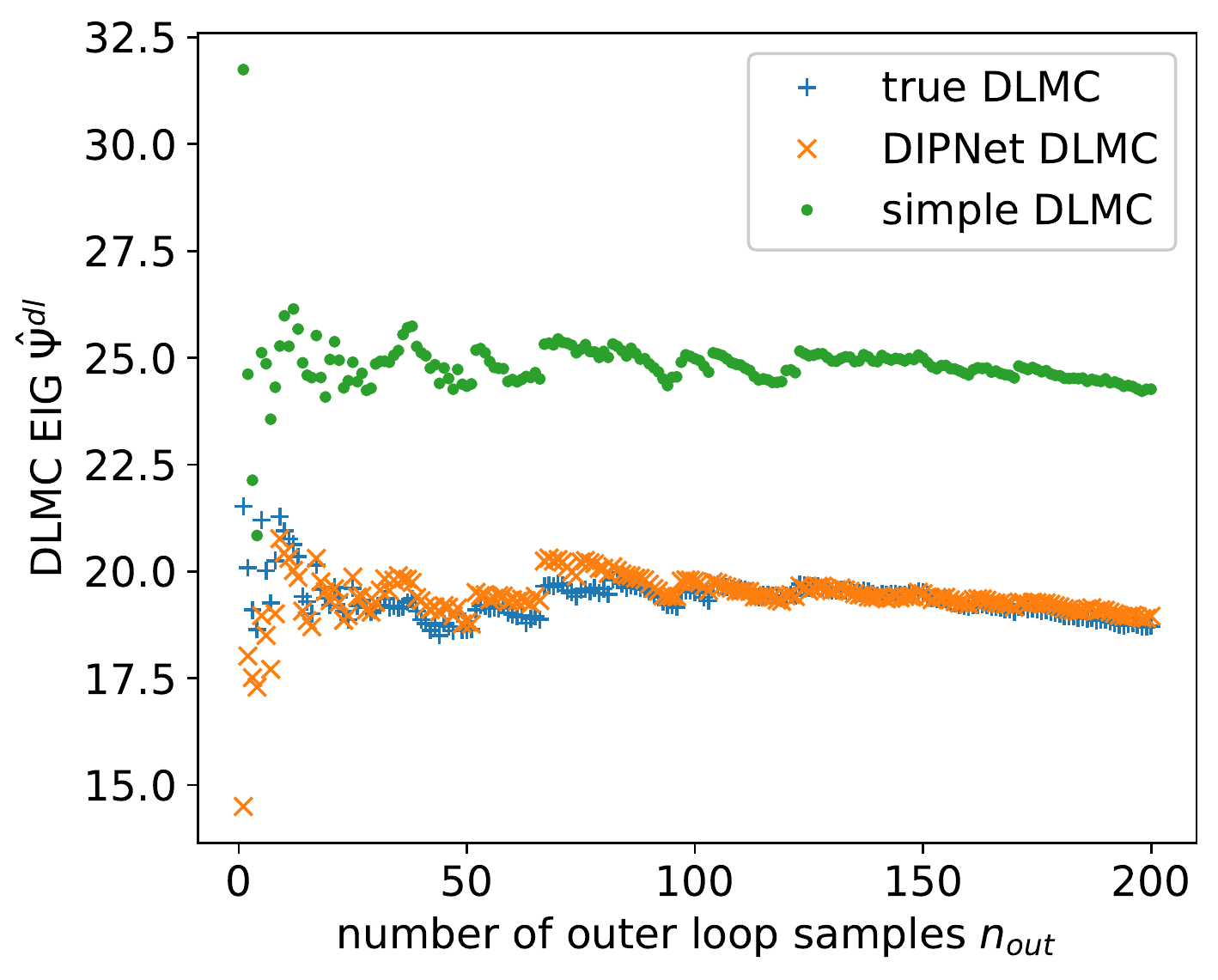}}
\caption{EIG of true DLMC ($n_{\text{in}} = 60000$), DIPNet DLMC ($n_{\text{in}} = 60000$), and simple DLMC ($n_{\text{in}} = 5757$) with increasing number of outer loop samples $n_{\text{out}}$ at $3$ random designs.}
\label{fig:h3}
\end{center}
\vskip -0.2in
\end{figure}

\kw{To show the effectiveness of truncated rank (breadth) for DIPNet surrogate, we evaluate the log normalization constant $\log\pi(\obs)$  and EIG $\Psi$ with breadth $= 10,25,50,100$ and compared with true MC and simple MC in Fig.\ \ref{fig:h4}. We can see that with increasing breadth, the relative error is decreasing, but gets worse when breadth reaches $100$. With  breadth $=100$, the difficulties of neural network training start to dominate and diminish the accuracy. We can also see it in the right part of the figure. The relative error of EIG approximation reduces (close to) linearly with respect to the generalization error of the DIPNet approximation of the observables with network breadth $= 10, 25, 50$, which confirms the error analysis in Theorem \ref{thm:error}. However, when the breadth increases to $100$, the neural network becomes less accurate (without using more training data), leading to less accurate EIG approximation.}

\begin{figure}[ht]
% \vskip 0.2in
\begin{center}
\centerline{\includegraphics[width=.33\columnwidth]{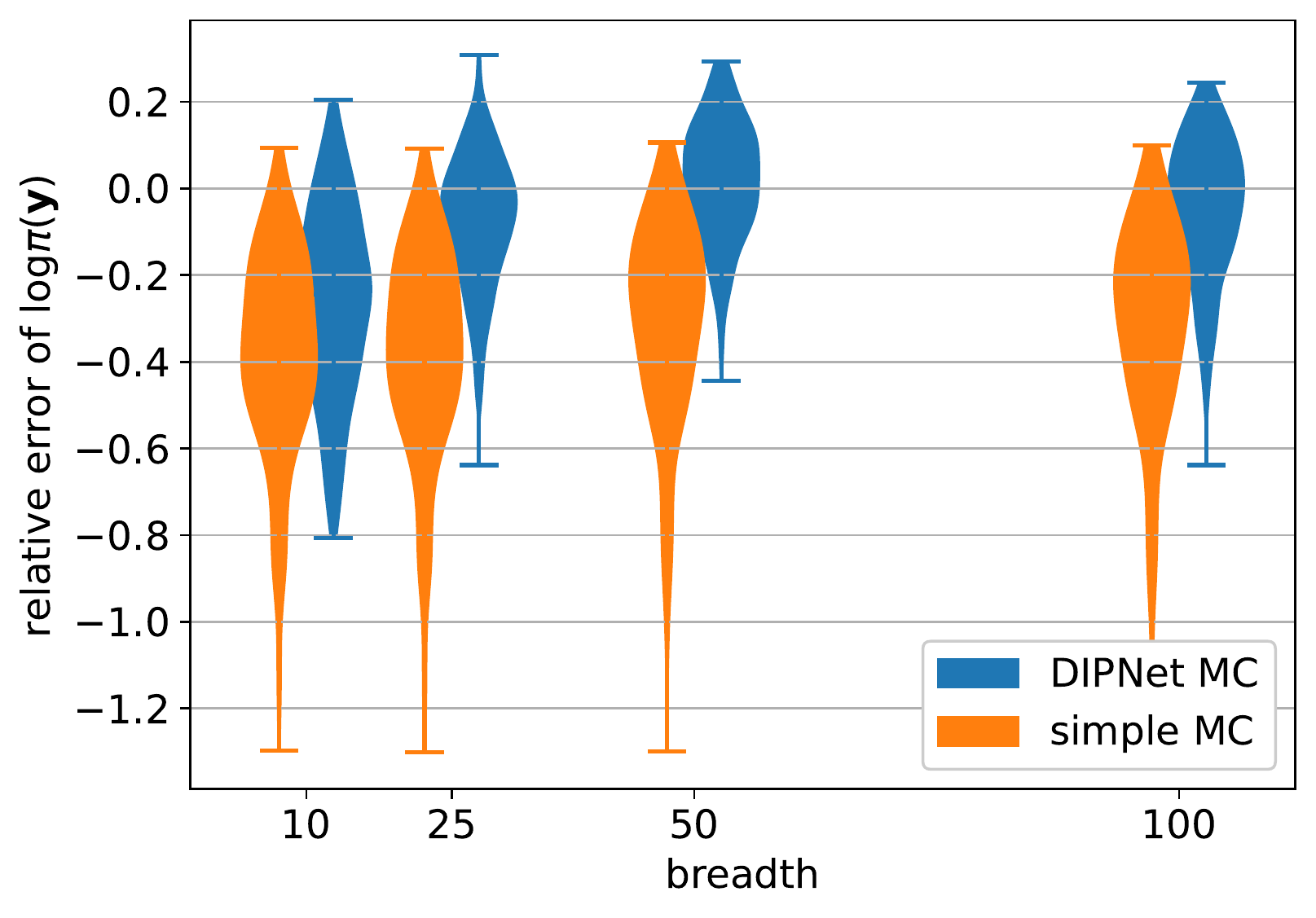}\includegraphics[width=.33\columnwidth]{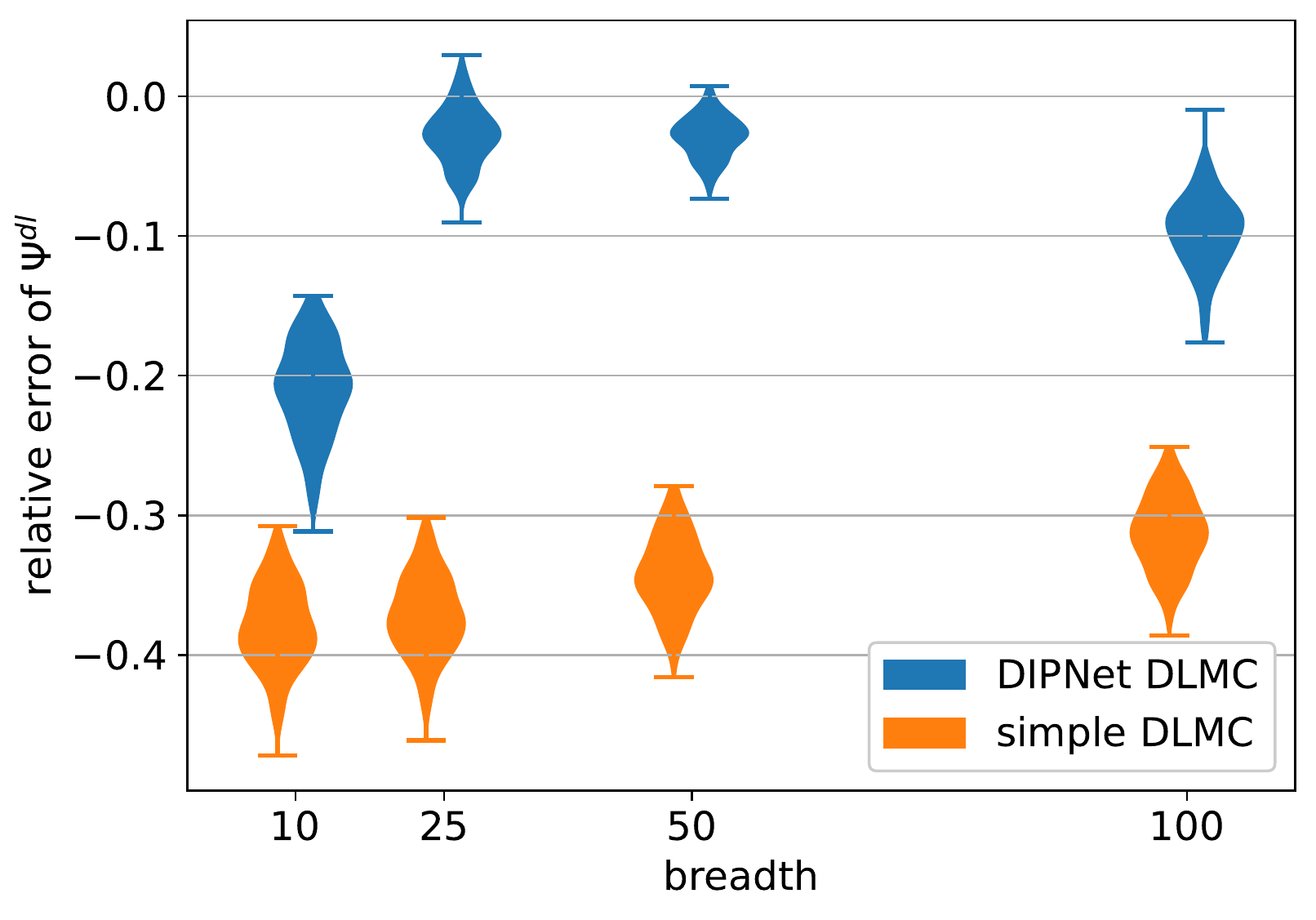}\includegraphics[width=.33\columnwidth]{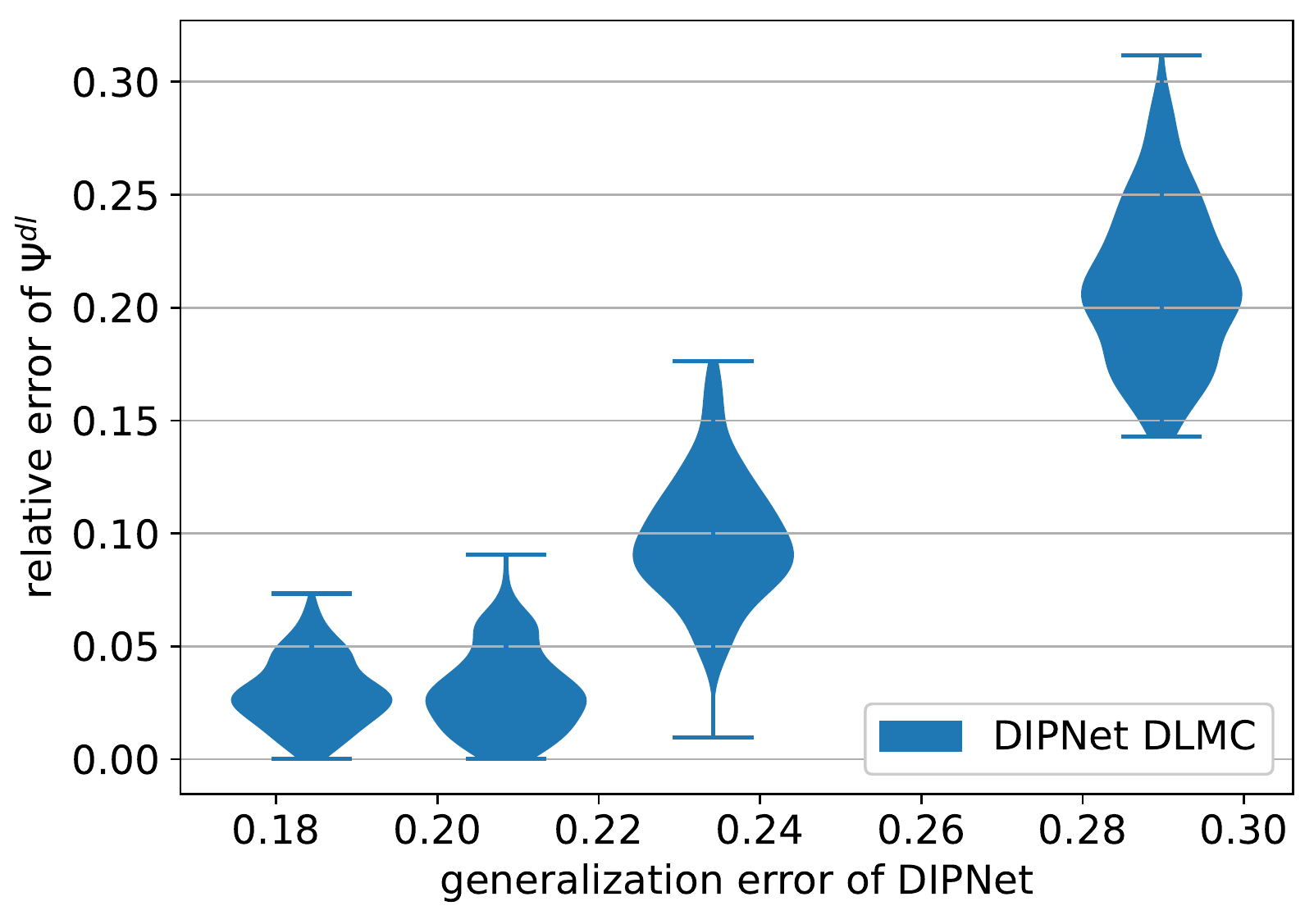}}
\caption{Sample distributions for 200 random designs of the relative approximation errors (compared to true MC) for the log normalization constant $\log\pi(\obs)$ (left) and EIG $\Psi$ (middle) by DIPNet MC and simple MC with increasing breadth. Right: Sample distributions for 200 random designs of the relative approximation errors (compared to true MC) against the corresponding DIPNet generalization error at breadth (from left to right) $ = 50, 100, 25, 10$.
}
\label{fig:h4}
\end{center}
\vskip -0.2in
\end{figure}

%%================================%%
%% convection diffusion problem%%
%%================================%%

\subsection{Advection-diffusion-reaction problem} \label{sec:confusion}

For the second numerical experiment we consider an OED problem for an advection-diffusion-reaction equation with a cubic nonlinear reaction term. The uncertain parameter $m$ appears as a log-coefficient of the cubic nonlinear reaction term. The PDE is defined in a domain $\Omega = (0,1)^2$ as 
\begin{subequations}
\begin{align}
  - \nabla \cdot (k \nabla u) + \mathbf{v} \cdot \nabla u + e^m u^3 &=
  f \quad \text{in } \Omega, \\ 
  u &= 0 \text{ on } \partial \Omega, \\
  \mathcal{F}_d(m) &= [u(\bs{x}_i,m)] \quad \text{at } \bs{x}_i \in \Omega.
\end{align}
\end{subequations}
Here $k=0.01$ is the diffusion coefficient.
The volumetric forcing function f is a smoothed Gaussian bump located at $\bs{x} = (0.7,0.7)$,
\begin{equation}
  f(\mathbf{x}) = \max\bigg( 0.5, e^{-25(x_1 - 0.7)^2  -25 (x_2 - 0.7)^2}\bigg).
\end{equation}
The velocity field $\mathbf{v}$ is a solution of a steady-state Navier Stokes equation with shearing boundary conditions driving the flow (see the Appendix in \cite{OLearyRoseberryVillaChenEtAl2022} for more information on the flow). The candidate sensor locations are located in a linearly spaced mesh-grid of points in $(0.1,0.9)\times(0.1,0.9)$, \kw{with coordinates $\{(0.1i, 0.1j), i,j = 1,2,\dots,9\}$}. The prior distribution for the uncertain parameter $m$ is a mean zero Gaussian with covariance $\Cpr = (I - 0.1\Delta)^{-2}$. The mesh used for this problem is uniform of size $128\times128$. We use linear elements for both $u$ and $m$, leading to a discrete parameter of dimension $16,641$. Fig.\ \ref{fig:c0} gives a prior sample of the parameter field $m$ and the solution $u$ with all $100$ candidate sensor locations in white circles.
\begin{figure}[ht]
\vskip 0.2in
\begin{center}
\centerline{\includegraphics[width=.45\columnwidth]{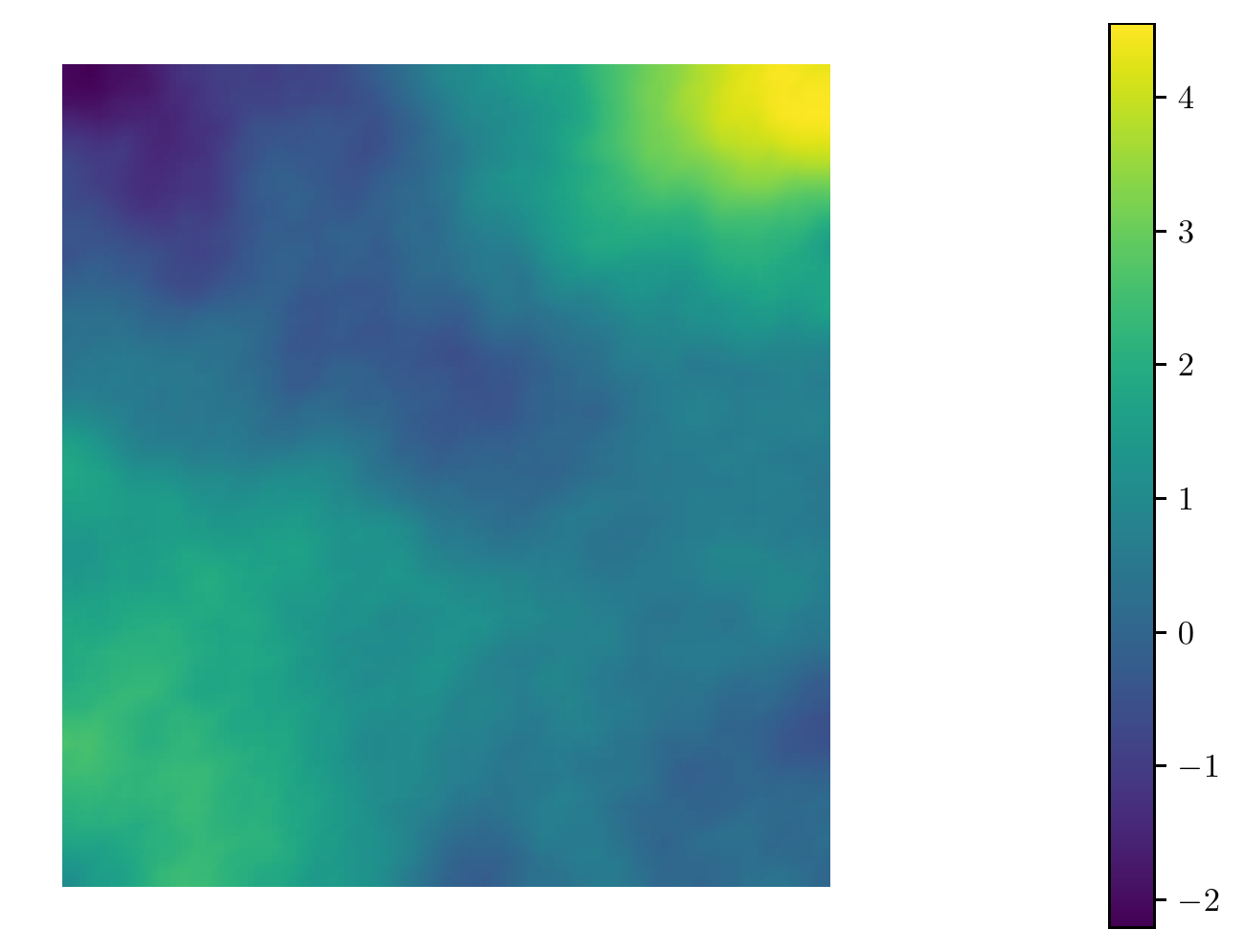}\includegraphics[width=.45\columnwidth]{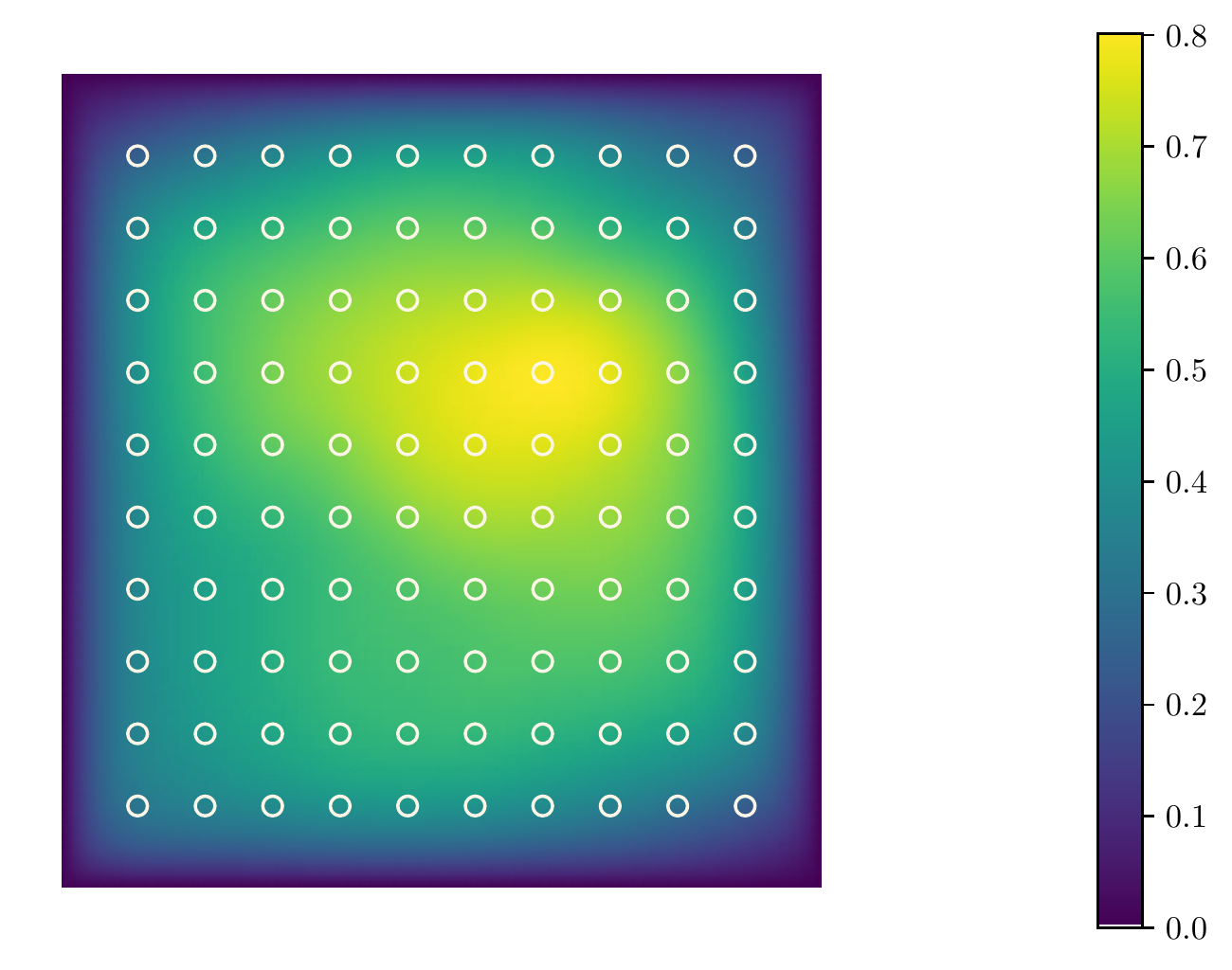}}
\caption{A random sample of the parameter field $\bipar$ (left) and the corresponding solution $\mathbf{u}$ with candidate observation sensor locations marked in circles (right) for the advection-diffusion-reaction problem.}
\label{fig:c0}
\end{center}
\vskip -0.2in
\end{figure}

The neural network surrogate is trained adaptively using $409$ training samples and $102$ validation samples. Using $512$ independent testing samples the DIPNet network was \tom{$97.13\% \ell^2$} accurate (see equation \ref{eq:ell2acc}). The network has $20$ low-rank residual layers, each with a layer rank of $10$, the breadth of the network is $r_M = r_F = 25$. The computational cost of the $25$ dimensional active subspace projector using $256$ samples is equivalent to the cost of $34$ additional training data. As was noted before when the PDE is nonlinear the linear adjoint-based derivative computations become much less of a computational burden. Thus we use \tom{$409+34 = 443$} samples for simple MC for fair comparison.

We first examine the log normalization constant $\log \pi(\obs)$ computed with our $409$ PDE-solve-based DIPNet surrogate compared against the truth computed with $60000$ MC samples and the simple MC computed with $443$ samples. Fig.\ \ref{fig:c1} shows the $\log \pi(\obs)$ comparison for three random designs that select $15$ sensors out of $100$ candidates. We can see that DIPNet MC converges to a value close to the true MC curves, while the simple MC's green star computed with the same number of  PDE solves ($443$) as DIPNet MC, has much worse accuracy. 

\begin{figure}[ht]
% \vskip 0.2in
\begin{center}
\centerline{\includegraphics[width=.33\columnwidth]{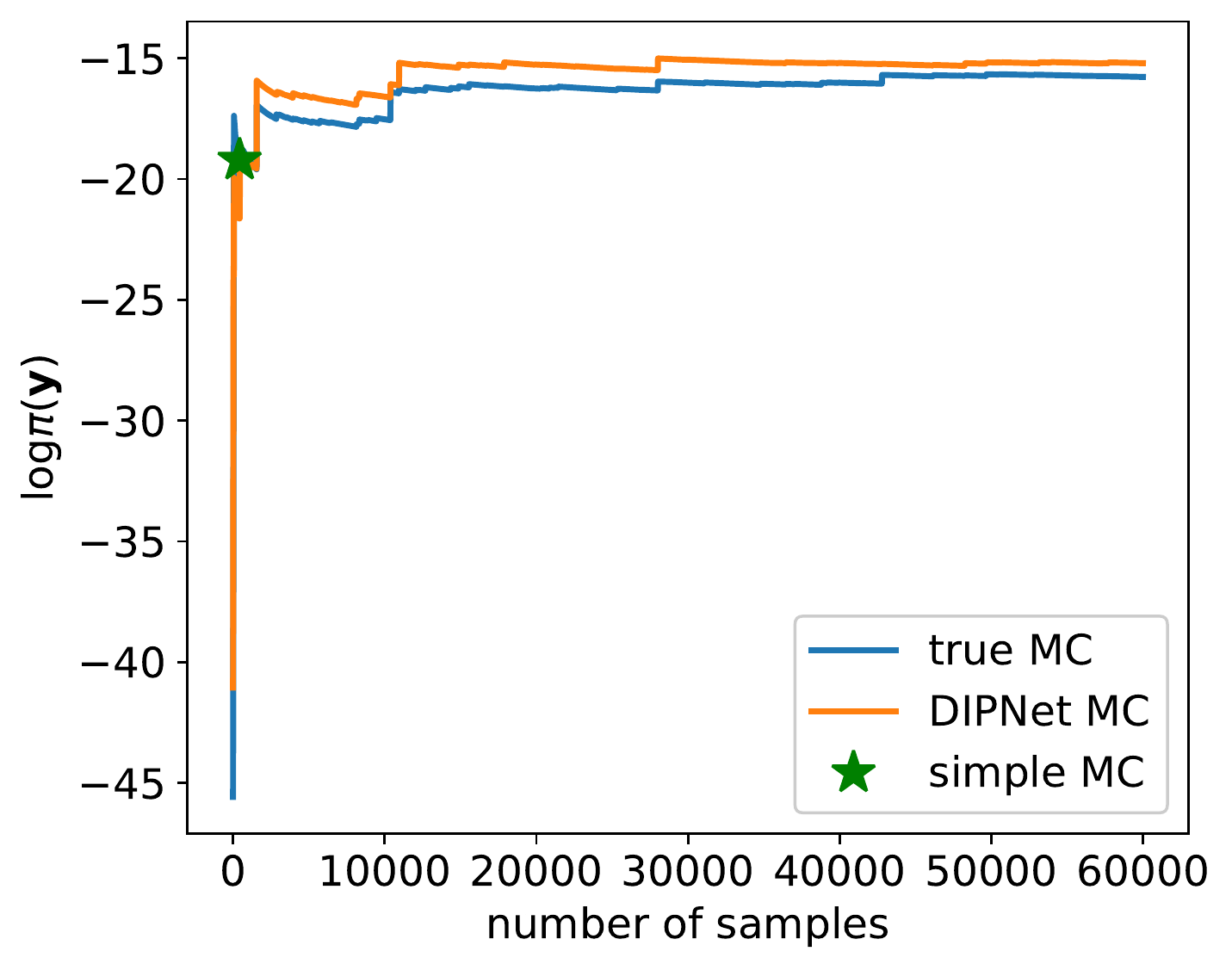}\includegraphics[width=.33\columnwidth]{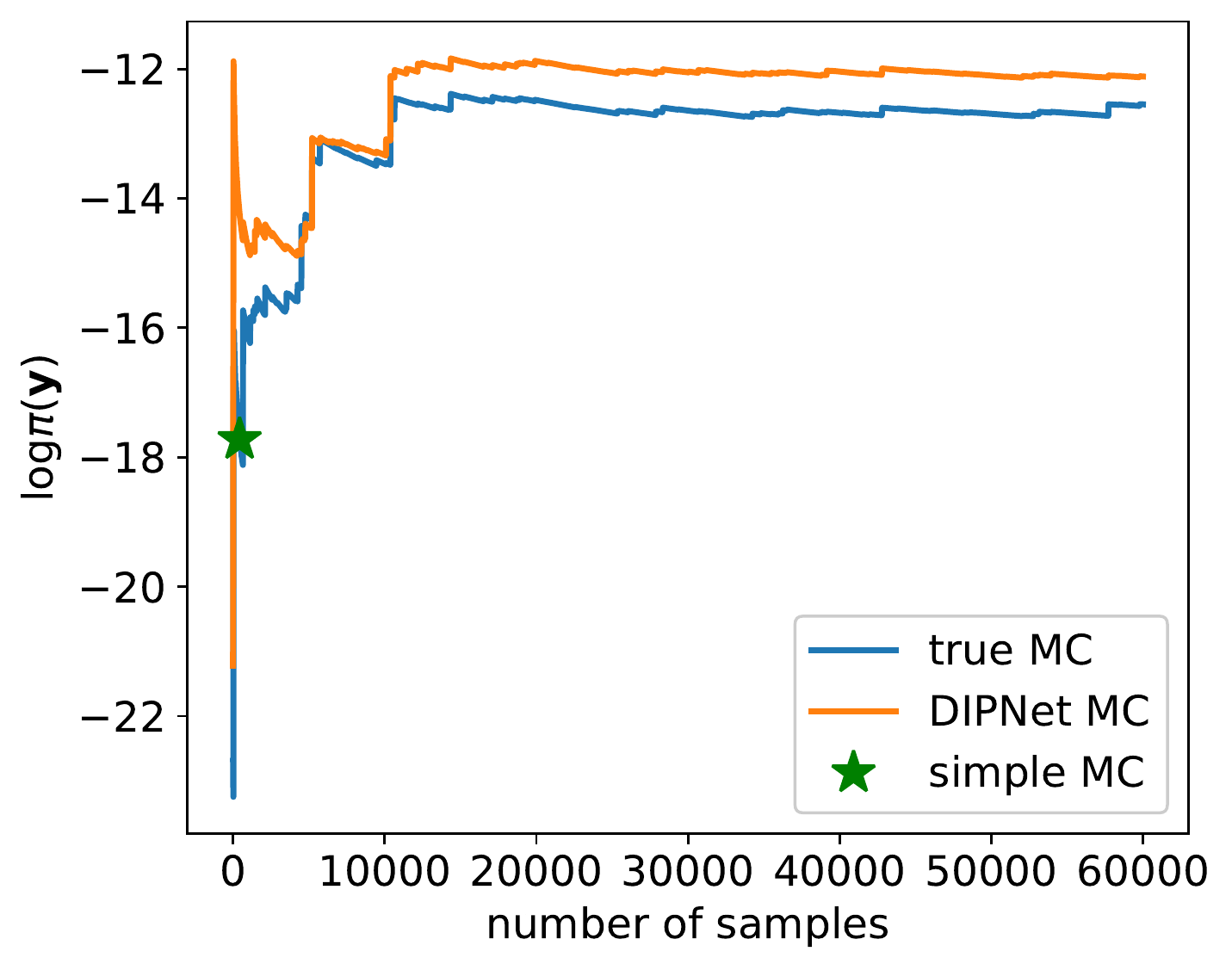}\includegraphics[width=.33\columnwidth]{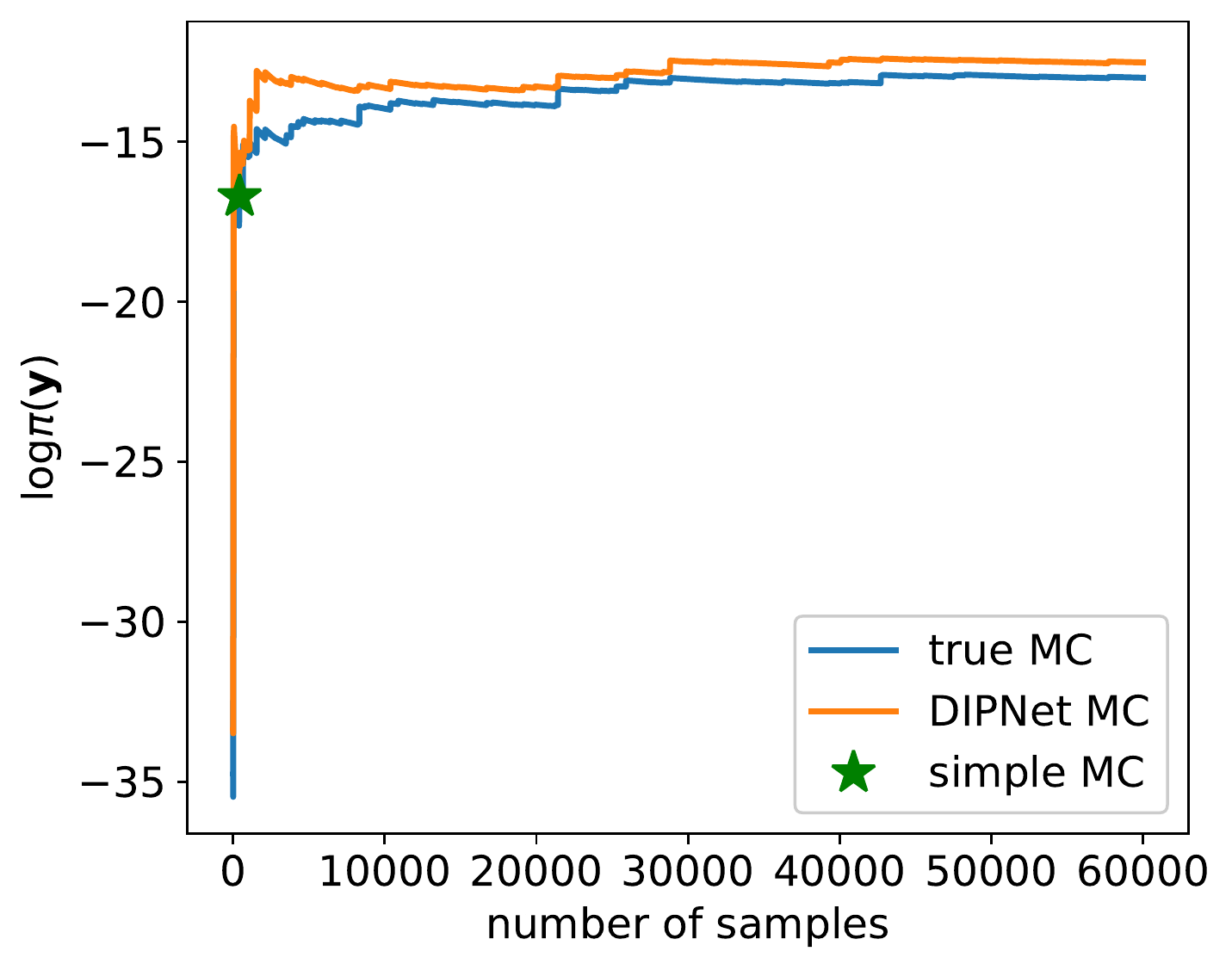}}
\caption{The approximation of the log normalization constant with increasing numbers of samples without (true MC) and with (DIPNet MC) surrogate at $3$ random designs. Green stars indicate $443$ samples (simple MC), having the same computational cost as DIPNet.}
\label{fig:c1}
\end{center}
\vskip -0.2in
\end{figure}

The left figure of Fig.\ \ref{fig:c2} shows the relative errors for $\log \pi(\obs)$ computed with the DIPNet surrogate and simple MC using $443$ samples based on the true MC with $60000$ samples, for $200$ random designs. We see again that DIPNet gives better accuracy with less bias than simple MC. 

 Fig.\ \ref{fig:c3} shows the EIG approximations of three random designs with increasing number of outer loop samples $n_{\text{out}}$ using the DIPNet MC with $60000$ (inner loop) samples, simple MC with $443$ samples, and true MC with $60000$ samples. We can see that the values of DIPNet MC are quite close to the true MC, while simple MC is far off. Relative errors of the EIG $\Psi$ ($n_{\text{out}} = 100$) computed with the DIPNet surrogate, the simple MC for  $200$ random designs
is given in the middle figure of Fig.\ \ref{fig:c2}. With the DIPNet DLMC $\Psi^{nn}$, we can use the greedy algorithm to find the optimal designs. The DIPNet greedy designs are presented as the pink crosses in the right figure of Fig.\ \ref{fig:c2}. We can see that the designs chosen by the greedy algorithm have much larger EIG values than all $200$ random designs.

\begin{figure}[ht]
\vskip 0.2in
\begin{center}
\centerline{\includegraphics[width=.33\columnwidth]{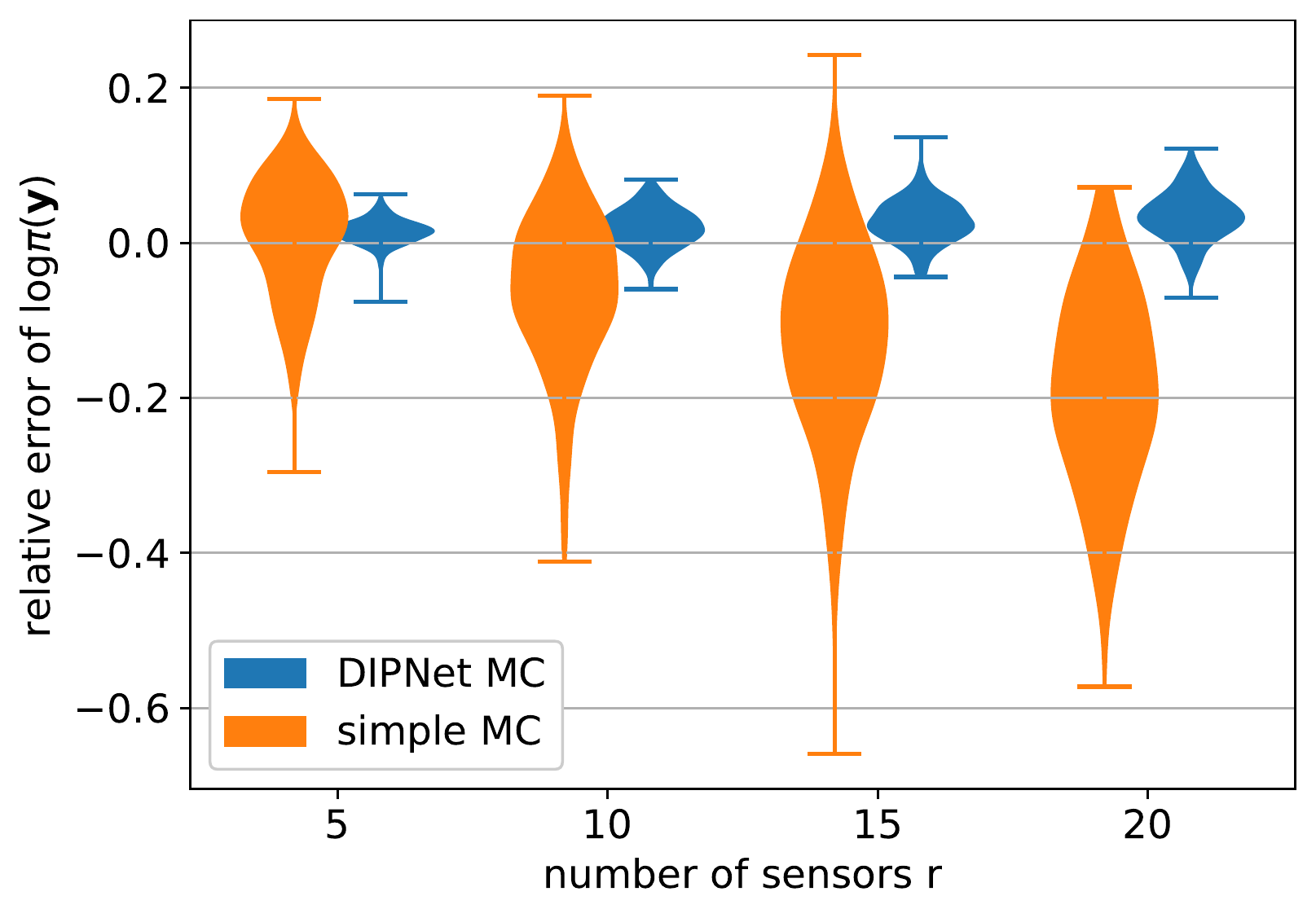}\includegraphics[width=.33\columnwidth]{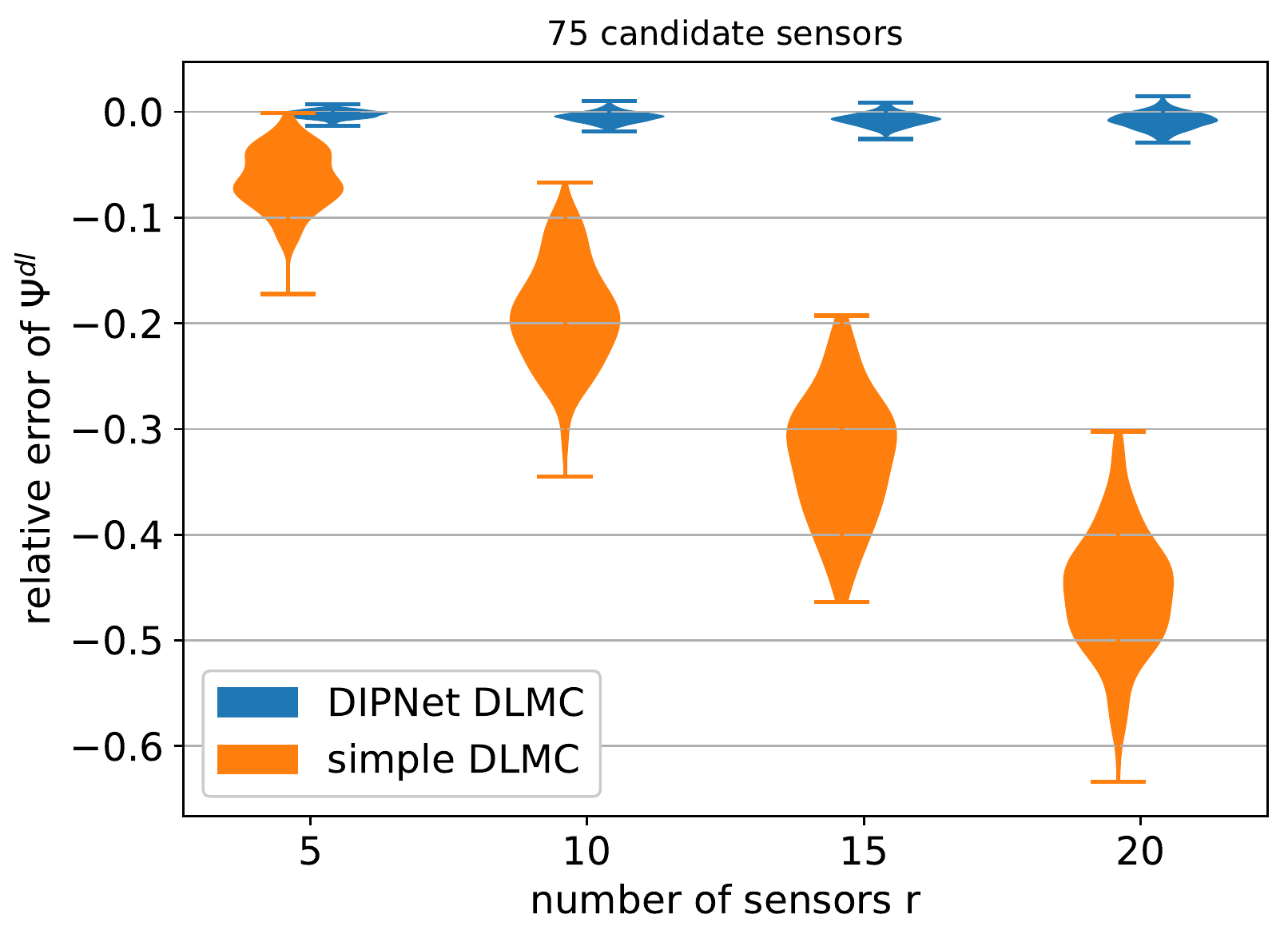}\includegraphics[width=.33\columnwidth]{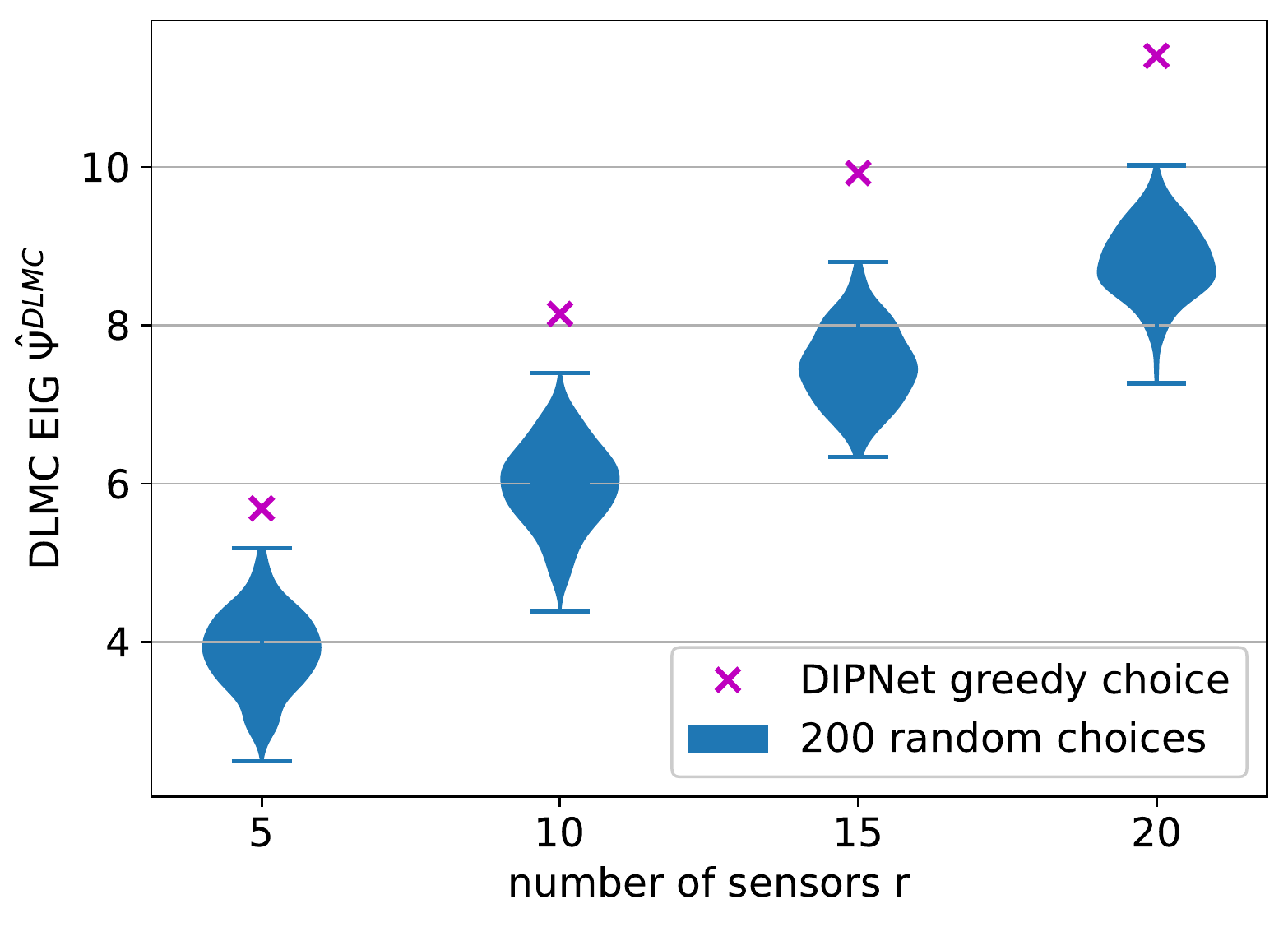}}
\caption{Sample distributions of the relative errors (compared to the truth MC) in approximating the log normalization constant $\log\pi(\obs)$ (left) and EIG $\Psi$ (middle) by DIPNet MC and simple MC with different number of sensors $r$; 
Right: Blue filled areas represent the sample distributions of the true DLMC EIG $\Psi^{dl}$ for 200 random designs. Pink crosses are the true DLMC EIG $\Psi^{dl}$ of designs chosen by the greedy optimization using DIPNet surrogates. }
\label{fig:c2}
\end{center}
\vskip -0.2in
\end{figure}

\begin{figure}[ht]
\vskip 0.2in
\begin{center}
\centerline{\includegraphics[width=.33\columnwidth]{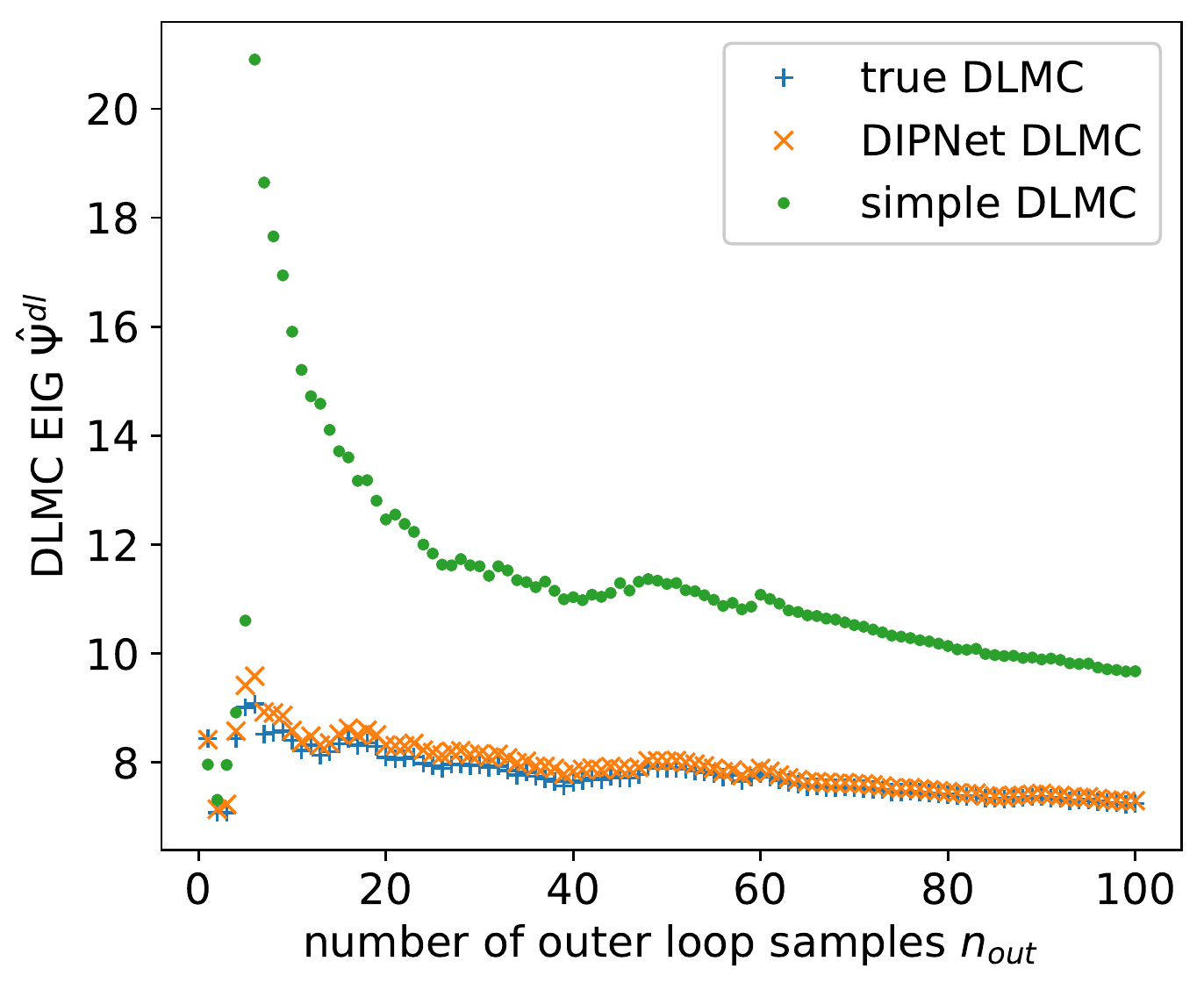}\includegraphics[width=.33\columnwidth]{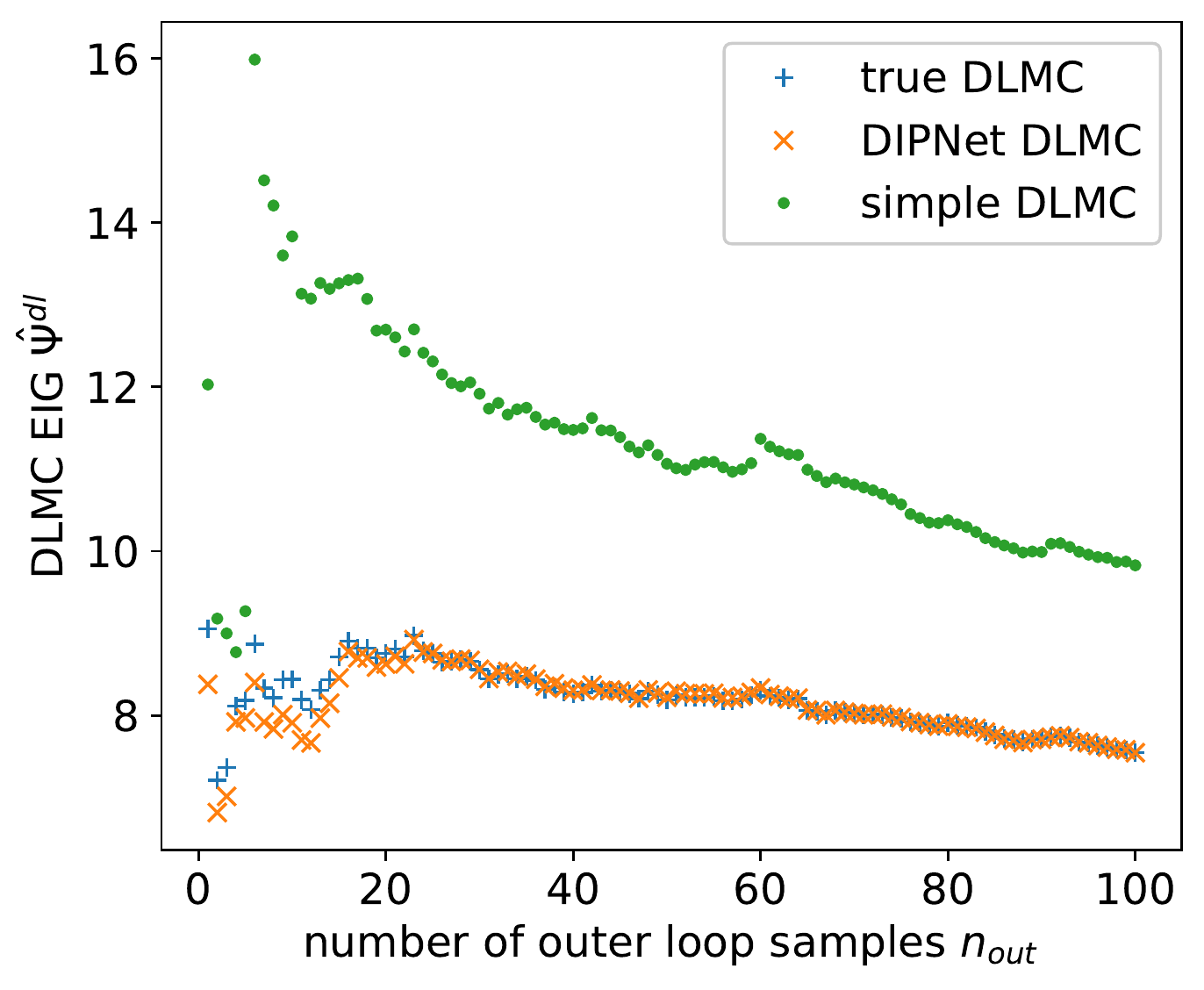}\includegraphics[width=.33\columnwidth]{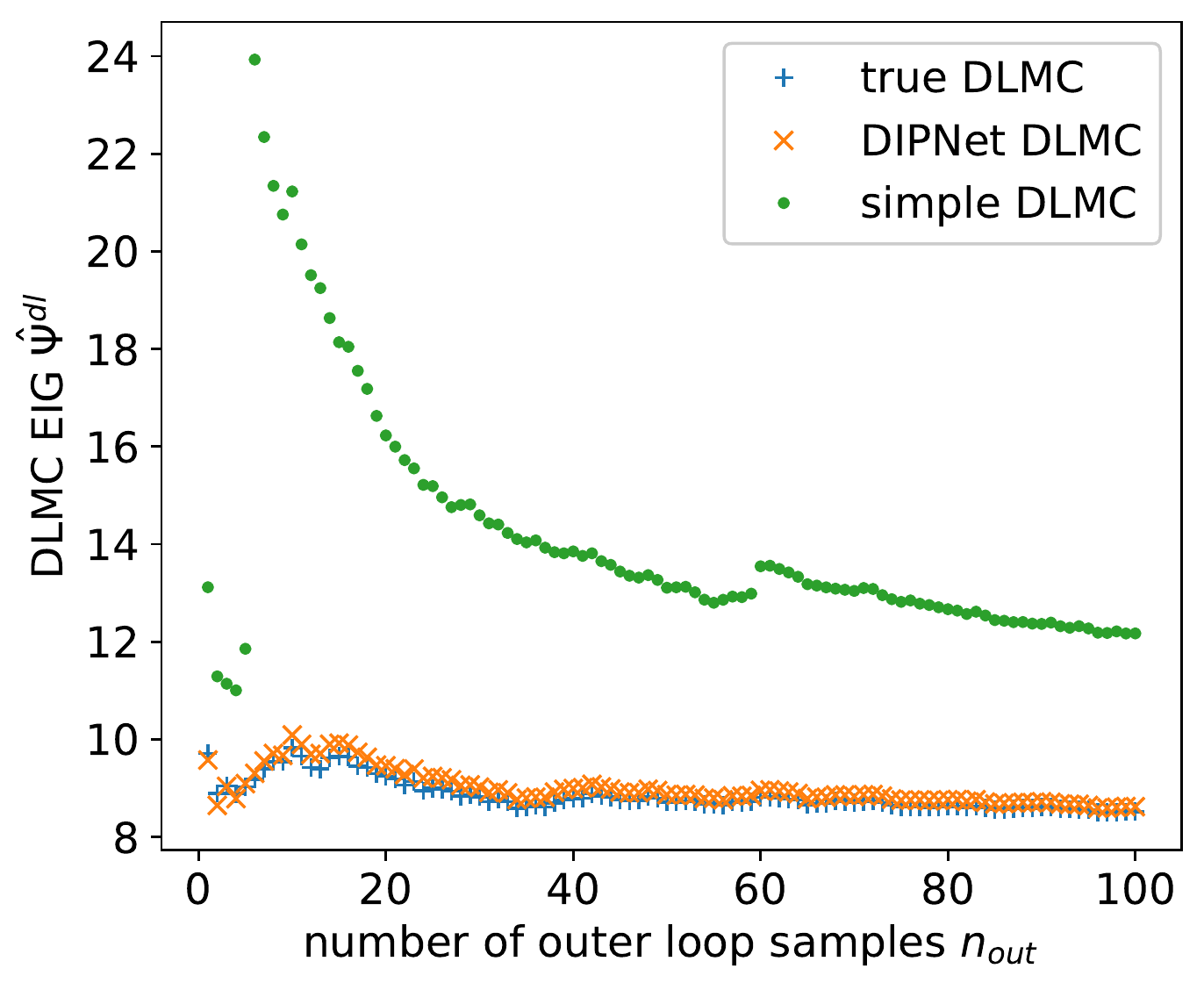}}
\caption{EIG of true DLMC ($n_{\text{in}} = 60000$), DIPNet DLMC ($n_{\text{in}} = 60000$) and simple DLMC ($n_{\text{in}} = 443$) with increasing number of outer loop samples $n_{\text{out}}$ at $3$ random designs.}
\label{fig:c3}
\end{center}
\vskip -0.2in
\end{figure}

% \section{Discussion}\label{sec12}

% Discussions should be brief and focused. In some disciplines use of Discussion or `Conclusion' is interchangeable. It is not mandatory to use both. Some journals prefer a section `Results and Discussion' followed by a section `Conclusion'. Please refer to Journal-level guidance for any specific requirements. 

\section{Conclusions}\label{sec:conclusions}

We have developed a computational method based on DIPNet surrogates for solving large-scale PDE-constrained Bayesian OED problems to determine optimal sensor locations (using the EIG criterion) to best infer infinite-dimensional parameters. We exploited the intrinsic low dimensionality of the parameter and data spaces and constructed a DIPNet surrogate for the parameter-to-observable map. The surrogate was used repeatedly in the evaluation of the normalization constant and the EIG. We presented error analysis for the approximation of the normalization constant and the EIG, showing that the errors are of the same order as the DIPNet RMS approximation error. Moreover, we used a greedy algorithm to solve the combinatorial optimization problem for sensor selection. The computational efficiency and accuracy of our approach are demonstrated by two numerical experiments. Future work will focus on gradient-based optimization also using the derivative information of the DIPNet w.r.t.\ both the parameter and the design variables, on the use of different optimality criteria such as A-optimality or D-optimality, and on exploring new network architectures for intrinsically high-dimensional Bayesian OED problems.

\bibliography{tom,references}% common bib file


\end{document}